\PassOptionsToPackage{pdfpagelabels=false}{hyperref}
\pdfoutput=1
\documentclass[copyright,creativecommons]{eptcs}
\usepackage{graphicx}
\usepackage{cmll}
\usepackage{amsmath,amsthm,amssymb}
\usepackage{mathtools}
\usepackage[UKenglish]{babel}
\usepackage[english=british]{csquotes}
\usepackage[compact]{titlesec}
\makeatletter
\let\c@author\relax
\makeatother
\usepackage[giveninits=true,url=false,date=year,dashed=false,bibstyle=authoryear]{biblatex}
\makeatletter
\input{numeric.bbx}
\makeatother
\bibliography{references}
\usepackage{cleveref}
\usepackage{adjustbox}
\usepackage[autosize,tikz]{dot2texi}
\usepackage[margin=10pt]{xsavebox}

\usepackage{rotating}
\allowdisplaybreaks[1]

\newtheorem{corollary}{Corollary}
\newtheorem{lemma}{Lemma}
\newtheorem{proposition}{Proposition}
\newtheorem{theorem}{Theorem}
\theoremstyle{definition}
\newtheorem{definition}{Definition}
\theoremstyle{remark}
\newtheorem{example}{Example}
\newtheorem{remark}{Remark}

\usepackage{tikz,pgfkeys}
\newcommand{\tinymorphism}[2][]{\smash{\ensuremath{\tikz[#1]{\scalecobordisms{0.35}\node[#2] (m) {};}}}}
\newcommand{\tinymorphisms}[1]{\ensuremath{\tikz[baseline=(current bounding box.center)]{\scalecobordisms{0.35}#1}}}

\usepackage{Cobordism}
\scalecobordisms{0.65}

\makeatletter
\def\calign@preamble{%
   &\hfil\strut@
    \setboxz@h{\@lign$\m@th\displaystyle{##}$}%
    \ifmeasuring@\savefieldlength@\fi
    \set@field
    \hfil
    \tabskip\alignsep@
}
\let\cmeasure@\measure@
\patchcmd\cmeasure@{\divide\@tempcntb\tw@}{}{}{}
\patchcmd\cmeasure@{\divide\@tempcntb\tw@}{}{}{}
\patchcmd\cmeasure@{\ifodd\maxfields@
  \global\advance\maxfields@\@ne
  \fi}{}{}{}
\newenvironment{calign}
{%
  \let\align@preamble\calign@preamble
  \let\measure@\cmeasure@
  \align
}
{%
  \endalign
}
\makeatother

\DeclareMathOperator{\Set}{\mathbf{Set}}
\DeclareMathOperator{\Rel}{\mathbf{Rel}}
\DeclareMathOperator{\Vect}{\mathbf{Vect}}
\DeclareMathOperator{\Cat}{\mathbf{Cat}}
\DeclareMathOperator{\Prof}{\mathbf{Prof}}
\DeclareMathOperator{\Tr}{\mathsf{Tr}}
\DeclareMathOperator{\NTr}{\mathsf{NTr}}
\DeclareMathOperator{\RTr}{\otimes \mathsf{Tr}_R}
\DeclareMathOperator{\RTrPar}{\invamp \mathsf{Tr}_R}
\DeclareMathOperator{\LTr}{\invamp \mathsf{Tr}_L}
\DeclareMathOperator{\LTrTensor}{\otimes \mathsf{Tr}_L}
\newcommand{\cat}[1]{\mathcal{#1}}

\newcommand{\op}[1]{{#1}^\mathrm{op}}
\newcommand{\ra}[1]{{#1}^\mathsf{RA}}
\title{Traced monoidal categories as algebraic structures in $\Prof$}
\author{
Nick Hu
  \institute{Department of Computer Science\\ University of Oxford\\
    Oxford, United Kingdom}
  \email{nick.hu@cs.ox.ac.uk}
\and
Jamie Vicary
  \institute{Department of Computer Science and Technology\\University of Cambridge\\
    Cambridge, United Kingdom}
  \email{jamie.vicary@cl.cam.ac.uk}
}
\def\titlerunning{Traced monoidal categories as algebraic structures in $\Prof$}
\def\authorrunning{N. Hu \& J. Vicary}
\begin{document}
\maketitle
\begin{abstract}
  We define a traced pseudomonoid as a pseudomonoid in a monoidal bicategory equipped with extra structure, giving a new characterisation of Cauchy complete traced monoidal categories as algebraic structures in $\Prof$, the monoidal bicategory of profunctors. This enables reasoning about the trace using the graphical calculus for monoidal bicategories, which we illustrate in detail. We apply our techniques to study traced $*$-autonomous categories, proving a new equivalence result between the left $\otimes$-trace and the right $\invamp$-trace, and describing a new condition under which traced $*$-autonomous categories become autonomous.
\end{abstract}

\begin{xlrbox}{id}
  \begin{tikzpicture}[baseline=(current bounding box.center)]
    \idstack{1}
  \end{tikzpicture}
\end{xlrbox}

\begin{xlrbox}{tallid}
  \begin{tikzpicture}[baseline=(current bounding box.center)]
    \idstack{2}
  \end{tikzpicture}
\end{xlrbox}

\begin{xlrbox}{verytallid}
  \begin{tikzpicture}[baseline=(current bounding box.center)]
    \idstack{3}
  \end{tikzpicture}
\end{xlrbox}

\begin{xlrbox}{veryverytallid}
  \begin{tikzpicture}[baseline=(current bounding box.center)]
    \idstack{4}
  \end{tikzpicture}
\end{xlrbox}

\begin{xlrbox}{double-id}
  \begin{tikzpicture}[baseline=(current bounding box.center)]
    \node[identity] (i1) {};
    \node[identity, xshift=-\cobwidth-\cobgap, anchor=top] (i2) at (i1.top) {};
  \end{tikzpicture}
\end{xlrbox}

\begin{xlrbox}{double-tallid}
  \begin{tikzpicture}[baseline=(current bounding box.center)]
    \node[identity] (i1) {};
    \node[identity, xshift=-\cobwidth-\cobgap, anchor=top] (i2) at (i1.top) {};
    \idstack[anchor=i1.bottom, direction=down]{1}
    \idstack[anchor=i2.bottom, direction=down]{1}
  \end{tikzpicture}
\end{xlrbox}

\begin{xlrbox}{double-verytallid}
  \begin{tikzpicture}[baseline=(current bounding box.center)]
    \node[identity] (i1) {};
    \node[identity, xshift=-\cobwidth-\cobgap, anchor=top] (i2) at (i1.top) {};
    \idstack[anchor=i1.bottom, direction=down]{2}
    \idstack[anchor=i2.bottom, direction=down]{2}
  \end{tikzpicture}
\end{xlrbox}

\begin{xlrbox}{monoid-white}
  \begin{tikzpicture}[baseline=(current bounding box.center)]
    \node[mult] {};
  \end{tikzpicture}
\end{xlrbox}

\begin{xlrbox}{tinymonoid-white}
  $\tinymorphism{mult}$
\end{xlrbox}

\begin{xlrbox}{monoid-black}
  \begin{tikzpicture}[baseline=(current bounding box.center)]
    \node[mult, dot=black] {};
  \end{tikzpicture}
\end{xlrbox}

\begin{xlrbox}{tinymonoid-black}
  \tinymorphism{mult, dot=black}
\end{xlrbox}

\begin{xlrbox}{comonoid-white}
  \begin{tikzpicture}[baseline=(current bounding box.center)]
    \node[comult] {};
  \end{tikzpicture}
\end{xlrbox}

\begin{xlrbox}{tinycomonoid-white}
  \tinymorphism{comult}
\end{xlrbox}

\begin{xlrbox}{comonoid-black}
  \begin{tikzpicture}[baseline=(current bounding box.center)]
    \node[comult, dot=black] {};
  \end{tikzpicture}
\end{xlrbox}

\begin{xlrbox}{tinycomonoid-black}
  \tinymorphism{comult, dot=black}
\end{xlrbox}

\begin{xlrbox}{unit-white}
  \begin{tikzpicture}[baseline=(current bounding box.center)]
    \node[unit] {};
  \end{tikzpicture}
\end{xlrbox}

\begin{xlrbox}{tinyunit-white}
   \tinymorphism[baseline={([yshift=-0.25\cobheight] m.center)}]{unit}
\end{xlrbox}

\begin{xlrbox}{unit-black}
  \begin{tikzpicture}[baseline=(current bounding box.center)]
    \node[unit, dot=black] {};
  \end{tikzpicture}
\end{xlrbox}

\begin{xlrbox}{tinyunit-black}
   \tinymorphism[baseline={([yshift=-0.25\cobheight] m.center)}]{unit, dot=black}
\end{xlrbox}

\begin{xlrbox}{counit-white}
  \begin{tikzpicture}[baseline=(current bounding box.center)]
    \node[counit] {};
  \end{tikzpicture}
\end{xlrbox}

\begin{xlrbox}{tinycounit-white}
  \tinymorphism{counit}
\end{xlrbox}

\begin{xlrbox}{counit-black}
  \begin{tikzpicture}[baseline=(current bounding box.center)]
    \node[counit, dot=black] {};
  \end{tikzpicture}
\end{xlrbox}

\begin{xlrbox}{tinycounit-black}
  \tinymorphism{counit, dot=black}
\end{xlrbox}

\section{Introduction}\label{sec:intro}

One way to interpret category theory is as a theory of \emph{systems} and
\emph{processes}, whereby monoidal structure naturally lends itself to enable
processes to be juxtaposed in parallel. Following this analogy, the
presence of a \emph{trace} structure embodies the notion of feedback:
some output of a process is directly fed back in to one of its inputs. For
instance, if we think of processes as programs, then feedback is some kind of
recursion \autocite{hasegawaRecursionCyclicSharing1997}. This becomes clearer
still when we consider how tracing is depicted in the standard graphical
calculus \autocite[\S~5]{selingerSurveyGraphicalLanguages2009}, as follows:
\begin{equation}\label{eq:trace}
\begin{tikzpicture}[baseline=(current bounding box.center)]
      \begin{scope}[internal string scope]
        \node [tiny label, draw=black, text=black] (f) {$f$};
        \node [xshift=\cobgap] (f') at (f) {};
        \node [black, yshift=\toff+0.5\cobheight, xshift=-\cobgap] at (f.north) {\tiny $A$};
        \node [black, yshift=-\boff-0.5\cobheight, xshift=-\cobgap] at (f.south) {\tiny $B$};
        \node [black, yshift=\toff+0.5\cobheight, xshift=\cobgap] at (f.north) {\tiny $X$};
        \node [black, yshift=-\boff-0.5\cobheight, xshift=\cobgap] at (f.south) {\tiny $X$};
        \draw [double=black] ([yshift=\toff+0.5\cobheight, xshift=-\cobgap] f.center)
          to [out=down, in=130] (f.north west)
          to (f.south west)
          to [out=-130, in=up] ([yshift=-\boff-0.5\cobheight, xshift=-\cobgap] f.center);
        \draw [double=black] ([yshift=\toff+0.5\cobheight, xshift=\cobgap] f.center)
          to [out=down, in=40] (f.north east)
          to (f.south east)
          to [out=-40, in=up] ([yshift=-\boff-0.5\cobheight, xshift=\cobgap] f.center);
      \end{scope}
    \end{tikzpicture}
    \leadsto
    \begin{tikzpicture}[baseline=(current bounding box.center), decoration={
        markings, mark=at position 0.475 with {\arrowreversed[black]{Stealth[length=1mm]}}
      }]
      \begin{scope}[internal string scope]
        \node [tiny label, draw=black, text=black] (f) {$f$};
        \node [xshift=\cobgap] (f') at (f) {};
        \node [black, xshift=3\cobgap] at (f) {\tiny $X$};
        \node [black, yshift=\toff+0.5\cobheight, xshift=-\cobgap] at (f.north) {\tiny $A$};
        \node [black, yshift=-\boff-0.5\cobheight, xshift=-\cobgap] at (f.south) {\tiny $B$};
        \draw [double=black] ([yshift=\toff+0.5\cobheight, xshift=-\cobgap] f.center)
          to [out=down, in=130] (f.north west)
          to (f.south west)
          to [out=-130, in=up] ([yshift=-\boff-0.5\cobheight, xshift=-\cobgap] f.center);
        \draw [double=black, postaction={decorate}] (f.center)
          to [out=40, in=down] ([yshift=0.5\cobheight, xshift=\cobgap] f.center)
          arc (180:0:0.5\cobgap)
          to [out=down, in=up] ([yshift=-0.5\cobheight, xshift=2\cobgap] f.center)
          arc (0:-180:0.5\cobgap)
          to [out=up, in=-40] (f.center);
      \end{scope}
    \end{tikzpicture}
\end{equation}
Many important algebraic structures which are typically defined
as sets-with-structure, like monoids, groups, or rings, may be described abstractly
as algebraic structure, which when
interpreted in $\Set$ yield the original definition. We call this process
\emph{externalisation}. The external version of a definition can then be
reinterpreted in a setting other than $\Set$ to expose meaningful connections
between known structures, or to generate new ones. For instance, a monoid in
$\Set$ is a standard monoid, but a monoid in $\Vect$ is a unital algebra,
and in $\Cat$ it is a strict monoidal category. Externalisation formalises the
relationship between these structures.

In this article, we externalise the $1$-categorical notion of traced monoidal
category, giving a new external definition of \emph{traced pseudomonoid}. We show that, when
interpreted in $\Prof$, the monoidal bicategory of categories and profunctors, this is equivalent to the standard definition of traced monoidal category. While the traditional definition of traced monoidal category has five separate axiom families, our traced pseudomonoid only has three, because two of the axioms become subsumed into the technology of $\Prof$. In this sense our externalised theory is a simpler than the traditional approach.

We make substantial use of the graphical calculus for compact closed
bicategories, categorifying the way one might use a PROP when working in a
symmetric monoidal theory \autocite{lackComposingPROPs2004}. $\Prof$
additionally admits a special string diagram calculus of \emph{internal string
diagrams} --- string diagrams \enquote{inside} string diagrams --- which we use
extensively to prove our results.

We apply our framework to derive new proofs of known facts about traced
monoidal categories in an entirely diagrammatic and synthetic way. For
instance, every braided autonomous category admits a trace, which we reduce to
the presence of a certain isomorphism. Following this, we proceed to analyse
the interaction between tracing and $*$-autonomous structure. We show that on a
$*$-autonomous category, a right $\otimes$-trace and a left $\invamp$-trace are
equivalent. We also derive an interesting sufficient condition for a traced
$*$-autonomous category to be compact closed, extending previous work of
\textcite{hajgatoTracedAutonomousCategories2013} which handled the symmetric
case.

\subsection{Related work}

Our traced pseudomonoid is a sort of \emph{categorification} of the standard
categorical notion of trace, as described by
\textcite{joyalTracedMonoidalCategories1996}. The idea of bicategories
\autocite{benabouIntroductionBicategories1967} as a formal arena for the study
of categories comes from \textcite{grayFormalCategoryTheory1974}, however an
issue which arises is that the obvious arena $\Cat$ preserves too little
information to study certain phenomena. Profunctors are one way to resolve this
\autocite{woodAbstractProArrows1982}, and furthermore they also naturally allow
for the diagrammatic methods we wish to employ. $\Prof$ is to $\Cat$ what
$\Rel$ is to $\Set$ \textcite[Example~5.1.5]{loregianCoendCalculus2019}.

Within the same framework, certain Frobenius pseudomonoids categorify the
notion of $*$-autonomous categories
\autocite{barrAutonomousCategories1979},
as first studied by \textcite{streetFrobeniusMonadsPseudomonoids2004}. We use
the term \enquote{$*$-autonomous} for the non-symmetric version, as described
by \textcite{barrNonsymmetricAutonomousCategories1995}. In a symmetric monoidal
category, the notions of trace and $*$-autonomous interact: a traced symmetric
$*$-autonomous monoidal category is compact closed
\autocite{hajgatoTracedAutonomousCategories2013}. An obvious conjecture is that
a traced $*$-autonomous category is autonomous (with left and right duals for
all objects), and in the last section of our paper we give an analysis of this
problem, deriving a sufficient condition for this result to hold.

\subsection{Outline}

In \Cref{sec:prof-string-diagrams}, we establish our technical background,
utilising the language of \emph{presentations}
\autocite[\S~2.10]{schommer-priesClassificationTwoDimensionalExtended2011} to
graphically represent different types of monoidal categories. Presentations
extending a pseudomonoid with right adjoints represent Cauchy complete monoidal
categories when interpreted in $\Prof$, and internal string diagrams
\autocite[\S~4]{bartlettModularCategoriesRepresentations2015} are also recalled
from existing literature.
\Cref{sec:traced-presentation} contains our main definition: the traced
pseudomonoid presentation. We show that its representations correspond exactly
to Cauchy complete traced monoidal categories
\autocite{joyalTracedMonoidalCategories1996}, using internal string diagrams as
our main proof technique.
\Cref{sec:autonomous-traced} illustrates using this framework that all Cauchy
complete braided autonomous categories are Cauchy complete traced.
\Cref{sec:*-autonomous-categories} concludes with a study on $*$-autonomous
categories, defined by the right-adjoint Frobenius pseudomonoid presentation
\autocite[\S~2.7]{dunnCoherenceFrobeniusPseudomonoids2016}, and their
interaction with tracing. We conjecture that every traced $*$-autonomous
category is autonomous, which is the non-symmetric generalisation of the result
of \textcite{hajgatoTracedAutonomousCategories2013}, and use our techniques to
give evidence for this conjecture.

\subsection{Acknowledgements}

The authors are grateful to Masahito Hasegawa for useful comments, and the
authors of \textcite{bartlettModularCategoriesRepresentations2015} for their
TikZ code for drawing internal string diagrams. The first author acknowledges
funding from the EPSRC [grant number EP/R513295/1].

\section{String diagrams and the bicategory of profunctors}\label{sec:prof-string-diagrams}

\subsection{Introduction}

In this section, we establish the definition of $\Prof$, and recall some of its
important properties. We also assume familiarity with \emph{string diagrams}
for compact closed categories, of the type described by
\textcite[\S~4.8]{selingerSurveyGraphicalLanguages2009}. There are two
main differences with our string diagrams:
\begin{enumerate}
  \item our string diagram convention is from bottom to top, rather than left
    to right;
  \item our setting is \emph{bicategorical}, which we view in projection. This
    means that, as usual, $0$-morphisms are represented by wire
    \emph{colourings}, and $1$-morphisms are represented by $2$-dimensional
    \emph{tiles} with some number of incoming wires and some number of outgoing
    wires, but in addition there are $2$-morphisms which are represented by
    wire-boundary-preserving (globular) \emph{rewrites} which act locally. In
    this context, the equational theory states that certain sequences of
    rewrites agree, and for each tile there is a \enquote{do nothing} rewrite
    corresponding to the identity $2$-morphism.
\end{enumerate}
This is the diagrammatic calculus of
\textcite{bartlettQuasistrictSymmetricMonoidal2014} enhanced with compact
structure, which means that $1$-morphisms may be rotated, changing the
orientation of their wires appropriately:
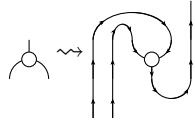
\[
\xusebox{monoid-white}
    \leadsto
    \begin{tikzpicture}[baseline=(current bounding box.center)]
      \node[comult, down] (c) {};
      \node[2Dcap, down, anchor=rightleg, span=0.5] (cap) at (c.leftleg) {};
      \node[2Dcap, down, anchor=rightleg, span=2] (cap') at (c.rightleg) {};
      \node[2Dcup, down, anchor=leftleg] (cup) at (c.bottom) {};
      \idstack[anchor=cap.leftleg, direction=down, orientation=up]{2}
      \idstack[anchor=cap'.leftleg, direction=down, orientation=up]{2}
      \idstack[anchor=cup.rightleg, direction=up, orientation=up]{2}
    \end{tikzpicture}
\]

Sometimes we will use colour-coded boxes to signal the local site at which a
$2$-morphism is being applied to aid the reader (for an example, see Definition~\ref{def:pseudomonoidpresentation}.) Additionally, we often use the
same symbol to denote a $2$-morphism, its inverse, or its adjoint mate; context
will disambiguate, but e.g.\ any $2$-morphism labelled $\alpha$ is morally the
associator move which type-checks, without significant additional nuance.

\begin{definition}[Bicategory of Profunctors {\autocite[Proposition~7.8.2]{borceuxHandbookCategoricalAlgebra1994}}]
  \emph{$\Prof$} is the bicategory of categories, profunctors, and natural
  transformations.
\end{definition}

\noindent
Additionally, $\Prof$ is compact closed in the sense of
\textcite[\S~2]{stayCompactClosedBicategories2013}, with the dual of $\cat{C}$ given by
$\op{\cat{C}}$; the structural information of this bicategory (the identity
profunctor, the symmetry, the co/unit of the compact structure, etc.) is given
by variations on the Hom-profunctor $\cat{C}(-, =)\colon \op{\cat{C}} \times
\cat{C} \to \Set$.

There exists an embedding theorem for $\Cat \to \Prof$.

\begin{lemma}\label[lemma]{prof-embedding}
  For each functor $F\colon \cat{C} \to \cat{D}$, there are associated
  profunctors $F_*\colon \op{\cat{C}} \times \cat{D} \to \Set$ and $F^*\colon
  \op{\cat{D}} \times \cat{C} \to \Set$ given by right and left actions on the
  Hom-functor $\cat{D}(-, =)$:
  \begin{equation*}
    F_* (d, c) = \cat{D}(F d, c),
    \quad
    F^* (c, d) = \cat{D}(c, F d).
  \end{equation*}
  Either mapping extends to an injective fully faithful pseudofunctor
  \autocite[Proposition~7.8.5]{borceuxHandbookCategoricalAlgebra1994}.
  Furthermore, $F^* \dashv F_*$ in $\Prof$
  \autocite[Proposition~7.9.1]{borceuxHandbookCategoricalAlgebra1994}.
\end{lemma}

$F^*$ is called the covariant embedding of $F$, or also its
\emph{representation}, and $F_*$ is called the contravariant embedding, or
alternatively its \emph{corepresentation}.

\Cref{prof-embedding} justifies the following condition, which we shall make
heavy use of throughout.

\begin{theorem}[{\autocite[Theorem~7.9.3]{borceuxHandbookCategoricalAlgebra1994}}]\label[theorem]{cauchy}
  Given a small category $\cat{C}$, the following conditions are equivalent:
  \begin{enumerate}
    \item $\cat{C}$ is Cauchy complete;
    \item for every small category $\cat{D}$, a profunctor $\op{\cat{C}} \times
      \cat{D} \to \Set$ has a right adjoint if and only if it is isomorphic to
      the covariant embedding of a functor.
  \end{enumerate}
\end{theorem}

\noindent
Informally, this describes when a profunctor is \enquote{the same} as a
functor. More precisely, it allows us to capture the conditions where we can
treat functors as profunctors (and conversely) --- when some profunctor $P$ is
(isomorphic to) the representation of some functor $F$ --- i.e.\ it justifies the
move from doing formal category theory in $\Cat$ to $\Prof$. Thus we must
qualify that throughout this article, our object of study is Cauchy complete
categories. Note that every category admits a universal embedding into its
Cauchy completion via the Karoubi envelope construction
\autocite{borceuxCauchyCompletionCategory1986}.

\subsection{Planar monoidal categories}

We recall the definitions of the pseudomonoid presentation and its
right-adjoint analogue \autocite{dunnCoherenceFrobeniusPseudomonoids2016}.

\begin{definition}
\label{def:pseudomonoidpresentation}
  \begin{xlrbox}{associator-1}
    \begin{tikzpicture}[baseline=(current bounding box.center)]
      \node[mult] (m1) {};
      \node[mult, anchor=top] (m2) at (m1.leftleg) {};
      \node[identity, anchor=top] (i) at (m1.rightleg) {};
      \node[mult, anchor=top] (m3) at (m2.leftleg) {};
      \node[identity, anchor=top] at (i.bottom) {};
      \node[identity, anchor=top] at (m2.rightleg) {};
      \node[draw=red, dashed, fit=(m2.top) (m2.rightleg) (m3.leftleg)]{};
      \node[draw=blue, dashed, fit=(m1.top) (m2.leftleg) (m1.rightleg)]{};
    \end{tikzpicture}
  \end{xlrbox}

  \begin{xlrbox}{associator-2}
    \begin{tikzpicture}[baseline=(current bounding box.center)]
      \node[mult, span=1.5] (m1) {};
      \node[mult, anchor=top] (m2) at (m1.leftleg) {};
      \node[identity, anchor=top] (i) at (m1.rightleg) {};
      \node[mult, anchor=top] at (m2.rightleg) {};
      \node[identity, anchor=top] at (i.bottom) {};
      \node[identity, anchor=top] at (m2.leftleg) {};
      \node[draw=red, dashed, fit=(m1.top) (m1.rightleg) (m2.leftleg)]{};
    \end{tikzpicture}
  \end{xlrbox}

  \begin{xlrbox}{associator-3}
    \begin{tikzpicture}[baseline=(current bounding box.center)]
      \node[mult, span=1.5] (m1) {};
      \node[mult, anchor=top] (m2) at (m1.rightleg) {};
      \node[identity, anchor=top] (i) at (m1.leftleg) {};
      \node[mult, anchor=top] (m3) at (m2.leftleg) {};
      \node[identity, anchor=top] at (i.bottom) {};
      \node[identity, anchor=top] at (m2.rightleg) {};
      \node[draw=red, dashed, fit=(m2.top) (m2.rightleg) (m3.leftleg)]{};
    \end{tikzpicture}
  \end{xlrbox}

  \begin{xlrbox}{associator-4}
    \begin{tikzpicture}[baseline=(current bounding box.center)]
      \node[mult, span=1.5] (m1) {};
      \node[mult, anchor=top] (m2) at (m1.leftleg) {};
      \node[mult, anchor=top] at (m1.rightleg) {};
      \node[draw=blue, dashed, fit=(m1.top) (m2.leftleg) (m1.rightleg)]{};
    \end{tikzpicture}
  \end{xlrbox}

  \begin{xlrbox}{associator-5}
    \begin{tikzpicture}[baseline=(current bounding box.center)]
      \node[mult] (m1) {};
      \node[mult, anchor=top] (m2) at (m1.rightleg) {};
      \node[identity, anchor=top] (i) at (m1.leftleg) {};
      \node[mult, anchor=top] at (m2.rightleg) {};
      \node[identity, anchor=top] at (i.bottom) {};
      \node[identity, anchor=top] at (m2.leftleg) {};
    \end{tikzpicture}
  \end{xlrbox}

  \begin{xlrbox}{unitor-l}
    \begin{tikzpicture}[baseline=(current bounding box.center)]
      \node[mult] (m1) {};
      \node[mult, anchor=top] (m2) at (m1.rightleg) {};
      \node[identity, anchor=top] at (m2.rightleg) {};
      \node[unit, anchor=top] (u) at (m2.leftleg) {};
      \node[identity, anchor=top] (i) at (m1.leftleg) {};
      \node[identity, anchor=top] at (i.bottom) {};
      \node[draw=red, dashed, fit=(m2.top) (m2.rightleg) (u)]{};
    \end{tikzpicture}
  \end{xlrbox}

  \begin{xlrbox}{unitor-r}
    \begin{tikzpicture}[baseline=(current bounding box.center)]
      \node[mult] (m1) {};
      \node[mult, anchor=top] (m2) at (m1.leftleg) {};
      \node[identity, anchor=top] at (m2.leftleg) {};
      \node[unit, anchor=top] (u) at (m2.rightleg) {};
      \node[identity, anchor=top] (i) at (m1.rightleg) {};
      \node[identity, anchor=top] at (i.bottom) {};
      \node[draw=blue, dashed, fit=(m2.top) (m2.leftleg) (u)]{};
      \node[draw=red, dashed, fit=(m1.top) (m2.leftleg) (m1.rightleg)]{};
    \end{tikzpicture}
  \end{xlrbox}

  The \emph{pseudomonoid presentation} of $\mathcal{M} = (\cdot, \xusebox{tinymonoid-white},
  \xusebox{tinyunit-white})$\footnote{$\mathcal{M}$ is not a bicategory, rather
  it is data from which a free symmetric monoidal bicategory can be generated à
  la generators-and-relations.} is given by
  \begin{itemize}
    \item a generating $0$-morphism: $\cdot$\footnote{The point $\cdot$ represents a
      $0$-dimensional aspect of our graphical calculus, i.e.\ the
    \enquote{colour} of the wires at the boundaries of $2$-dimensional tiles.
    Graphically, this \enquote{colour} is depicted by (implicitly) upwards
    flowing wires, contrasting with its dual colour, the downwards flowing
    wires.};
    \item generating $1$-morphisms: $\xusebox{tinymonoid-white}$ and $\xusebox{tinyunit-white}$;
    \item invertible generating $2$-morphisms expressing associativity and unitality respectively:
      
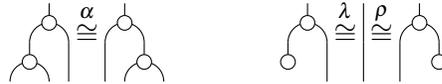
\begin{equation*}
        \begin{tikzpicture}[baseline=(current bounding box.center)]
          \node[mult] (m) {};
          \node[mult, anchor=top] at (m.leftleg) {};
          \node[identity, anchor=top] at (m.rightleg) {};
        \end{tikzpicture}
        \overset{\alpha}{\cong}
        \begin{tikzpicture}[baseline=(current bounding box.center)]
          \node[mult] (m) {};
          \node[mult, anchor=top] at (m.rightleg) {};
          \node[identity, anchor=top] at (m.leftleg) {};
        \end{tikzpicture}
        \qquad \qquad
        \begin{tikzpicture}[baseline=(current bounding box.center)]
          \node[mult] (m) {};
          \node[unit, anchor=top] at (m.leftleg) {};
          \node[identity, anchor=top] at (m.rightleg) {};
        \end{tikzpicture}
        \overset{\lambda}{\cong}
        \begin{tikzpicture}[baseline=(current bounding box.center)]
          \node[identity] (i) {};
          \node[identity, anchor=top] at (i.bottom) {};
        \end{tikzpicture}
        \overset{\rho}{\cong}
        \begin{tikzpicture}[baseline=(current bounding box.center)]
          \node[mult] (m) {};
          \node[unit, anchor=top] at (m.rightleg) {};
          \node[identity, anchor=top] at (m.leftleg) {};
        \end{tikzpicture}
      \end{equation*}
    \item equations witnessing that these inverses are coherent
      (pentagon\footnote{Due to the weak interchange structure of $\Prof$, this
      pentagon should technically be a hexagon where along the bottom, the
      right monoid and the left monoid are interchanged between the two
      associator moves, however for clarity we elide trivial interchange steps
      throughout.} and triangle equations):
      
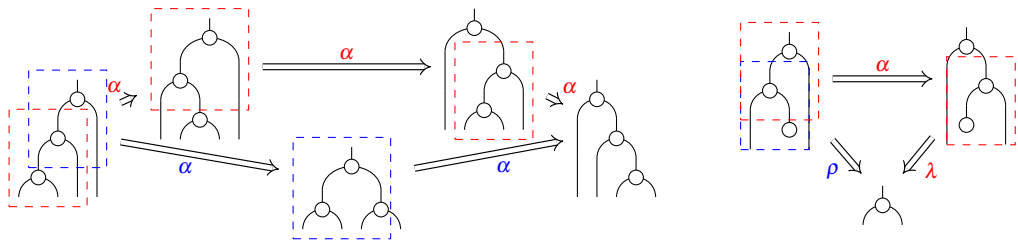
\begin{equation*}
        \begin{tikzcd}[math mode=false, row sep=-35pt, column sep=tiny]
          &
          \xusebox{associator-2}
          \arrow[rr, Rightarrow, "\textcolor{red}{$\alpha$}"]
          &&
          \xusebox{associator-3}
          \arrow[rd, Rightarrow, "\textcolor{red}{$\alpha$}"]
          & \\
          \xusebox{associator-1}
          \arrow[ru, Rightarrow, "\textcolor{red}{$\alpha$}"]
          \arrow[rrd, Rightarrow, "\textcolor{blue}{$\alpha$}"']
          &&&&
          \xusebox{associator-5}
          \\
          &&
          \xusebox{associator-4}
          \arrow[rru, Rightarrow, "\textcolor{blue}{$\alpha$}"']
          &&
        \end{tikzcd}
        \qquad
        \begin{tikzcd}[math mode=false, sep=tiny]
          \xusebox{unitor-r}
          \arrow[rr, Rightarrow, "\textcolor{red}{$\alpha$}"]
          \arrow[rd, Rightarrow, "\textcolor{blue}{$\rho$}"']
          &&
          \xusebox{unitor-l}
          \arrow[ld, Rightarrow, "\textcolor{red}{$\lambda$}"]
          \\
          &
          \xusebox{monoid-white}
          &
        \end{tikzcd}
      \end{equation*}
  \end{itemize}
\end{definition}

We actually use \emph{oriented} string diagrams, in the sense of
\autocite[\S~4]{selingerSurveyGraphicalLanguages2009}, as the dual of
$\cdot$, is given by the duality of $\cat{C}$ versus
$\op{\cat{C}}$ in $\Prof$ and is represented diagrammatically by
downwards-oriented strings. For cleanliness, we omit decorations for
upwards-oriented strings.

Presentations can be interpreted in a target symmetric monoidal bicategory, as follows.
\begin{definition}
  An \emph{interpretation} of a presentation $\mathcal{P}$ in a symmetric monoidal bicategory
  $\cat{C}$ is given by a strict symmetric monoidal 2-functor from the free
  symmetric monoidal bicategory on $\mathcal{P}$ to $\cat{C}$.
\end{definition}

\noindent
As discussed in \textcite[\S~2.1]{bartlettModularCategoriesRepresentations2015}, such an interpretation corresponds exactly to choosing for each $k$-dimensional generator of $\mathcal{P}$ a corresponding $k$-morphism of $\mathcal{C}$, satisfying the corresponding equations. So interpretations of presentations are easy to work with. The following then follows.

\begin{lemma}
  Interpretations of $\mathcal{M}$ in $\Prof$ correspond to Cauchy complete
  promonoidal categories.
\end{lemma}

\noindent
A promonoidal category is \enquote{nearly} a monoidal category: it captures
only when Hom-sets have the form $\cat{C} (X, Y \otimes Z)$ --- that is,
$\otimes$ may only appear as a right action on the Hom. To overcome this
limitation, we must restrict our attention to \emph{representable} profunctors
(equivalently, profunctors which admit right adjoints).

\begin{definition}[Free right-adjoint extension]
  For a presentation $\mathcal{P}$, we denote $\ra{\mathcal{P}}$ as the
  presentation with all of the data of $\mathcal{P}$, and in addition, for each
  generating $1$-morphism of $\mathcal{P}$: a freely-added right-adjoint
  $1$-morphism, unit and counit $2$-morphisms, and equational structure
  witnessing that the triangle equations for this adjunction hold.

  The procedure $\ra{(-)}$ which freely adds right adjoints is well-behaved
  in the sense of
  \textcite[\S~2.3]{bartlettModularCategoriesRepresentations2015}, and in
  general, given the data of $\mathcal{P}$, it is unambiguous to discuss a
  presentation $\ra{\mathcal{P}}$ with freely-added right adjoints without
  giving an explicit description as we do in \Cref{right-adjoint-monoid}.
\end{definition}

\begin{example}\label[example]{right-adjoint-monoid}
  \begin{xlrbox}{eta-prime-target}
    \begin{tikzpicture}[baseline=(current bounding box.center)]
      \node[comult, dot=black] (c) {};
      \node[mult, anchor=top] (m) at (c.bottom) {};
    \end{tikzpicture}
  \end{xlrbox}

  \begin{xlrbox}{eta-target}
    \begin{tikzpicture}[baseline=(current bounding box.center)]
      \node[comult] (c) {};
      \node[mult, dot=black, anchor=top] (m) at (c.bottom) {};
    \end{tikzpicture}
  \end{xlrbox}

  \begin{xlrbox}{white-monoid-triangle-1}
    \begin{tikzpicture}[baseline=(current bounding box.center)]
      \node[mult] (m) {};
      \node[draw=red, dashed, fit=(m.leftleg) (m.rightleg)]{};
    \end{tikzpicture}
  \end{xlrbox}

  \begin{xlrbox}{white-monoid-triangle-2}
    \begin{tikzpicture}[baseline=(current bounding box.center)]
      \node[mult] (m) {};
      \node[comult, dot=black, anchor=leftleg] (c) at (m.leftleg) {};
      \node[mult, anchor=top] (m') at (c.bottom) {};
      \node[draw=red, dashed, fit=(m.top) (m.leftleg) (m.rightleg) (c)]{};
    \end{tikzpicture}
  \end{xlrbox}

  \begin{xlrbox}{white-unit-triangle-1}
    \begin{tikzpicture}[baseline=(current bounding box.center)]
      \node[unit] (u) {};
      \node[draw=red, dashed, anchor=north, below=2\cobgap of u, minimum height=1cm, minimum width=1cm]{};
    \end{tikzpicture}
  \end{xlrbox}

  \begin{xlrbox}{white-unit-triangle-2}
    \begin{tikzpicture}[baseline=(current bounding box.center)]
      \node[counit, dot=black] (c) {};
      \node[unit, anchor=top] at (c.bottom) {};
      \node[unit, above=2\cobgap of c] (u) {};
      \node[draw=red, dashed, fit=(c.bottom) (u)]{};
    \end{tikzpicture}
  \end{xlrbox}

  \begin{xlrbox}{black-comonoid-triangle-1}
    \begin{tikzpicture}[baseline=(current bounding box.center)]
      \node[comult, dot=black] (c) {};
      \node[draw=red, dashed, fit=(c.leftleg) (c.rightleg)]{};
    \end{tikzpicture}
  \end{xlrbox}

  \begin{xlrbox}{black-comonoid-triangle-2}
    \begin{tikzpicture}[baseline=(current bounding box.center)]
      \node[comult, dot=black] (c) {};
      \node[mult, anchor=top] (m) at (c.bottom) {};
      \node[comult, dot=black, anchor=leftleg] (c') at (m.leftleg) {};
      \node[draw=red, dashed, fit=(m.top) (m.leftleg) (m.rightleg) (c'.bottom)]{};
    \end{tikzpicture}
  \end{xlrbox}

  \begin{xlrbox}{black-counit-triangle-1}
    \begin{tikzpicture}[baseline=(current bounding box.center)]
      \node[counit, dot=black] (c) {};
      \node[draw=red, dashed, anchor=south, above=2\cobgap of c, minimum height=1cm, minimum width=1cm]{};
    \end{tikzpicture}
  \end{xlrbox}

  \begin{xlrbox}{black-counit-triangle-2}
    \begin{tikzpicture}[baseline=(current bounding box.center)]
      \node[counit, dot=black] (c) {};
      \node[unit, above=2\cobgap of c] (u) {};
      \node[counit, dot=black, anchor=bottom] at (u.top) {};
      \node[draw=red, dashed, fit=(c) (u)]{};
    \end{tikzpicture}
  \end{xlrbox}

  The \emph{right-adjoint pseudomonoid presentation} $\ra{\mathcal{M}}$ is
  given by the data of $\mathcal{M}$, and additionally:
  \begin{itemize}
    \item $1$-morphisms: $\xusebox{tinycomonoid-black}$ and
      $\xusebox{tinycounit-black}$;
    \item unit and counit $2$-morphisms witnessing adjunctions
      $\xusebox{tinymonoid-white} \dashv \xusebox{tinycomonoid-black}$ and
      $\xusebox{tinyunit-white} \dashv \xusebox{tinycounit-black}$:
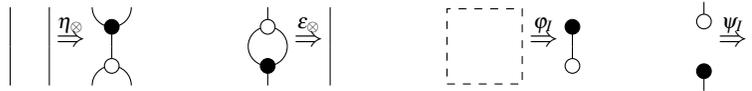
      \begin{equation*}
        \xusebox{double-tallid}
        \overset{\eta_\otimes}{\Rightarrow}
        \xusebox{eta-prime-target} \qquad\qquad
        \begin{tikzpicture}[baseline=(current bounding box.center)]
          \node[comult, dot=black] (c) {};
          \node[mult, anchor=leftleg] at (c.leftleg) {};
        \end{tikzpicture}
        \overset{\varepsilon_\otimes}{\Rightarrow}
        \xusebox{tallid} \qquad\qquad
        \begin{tikzpicture}[baseline=(current bounding box.center)]
          \node[draw, dashed, minimum height=1cm, minimum width=1cm] {};
        \end{tikzpicture}
        \overset{\varphi_I}{\Rightarrow}
        \begin{tikzpicture}[baseline=(current bounding box.center)]
          \node[counit, dot=black] (c) {};
          \node[unit, anchor=top] at (c.bottom) {};
        \end{tikzpicture} \qquad\qquad
        \begin{tikzpicture}[baseline=(current bounding box.center)]
          \node[counit, dot=black] (c) {};
          \node[unit, above=2\cobgap of c] {};
        \end{tikzpicture}
        \overset{\psi_I}{\Rightarrow}
        \xusebox{tallid}
      \end{equation*}
    \item equations witnessing that the adjunction is coherent (triangle equations):
      
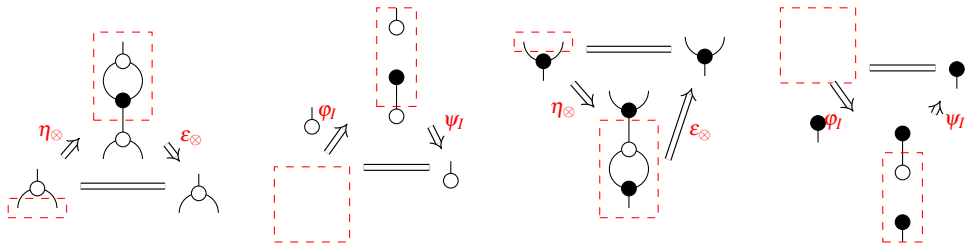
\begin{equation*}
        \begin{tikzcd}[math mode=false, sep=0pt, row sep=-5pt]
          &
          \xusebox{white-monoid-triangle-2}
          \arrow[rd, Rightarrow, "\textcolor{red}{$\varepsilon_\otimes$}"]
          &
          \\
          \xusebox{white-monoid-triangle-1}
          \arrow[ru, Rightarrow, "\textcolor{red}{$\eta_\otimes$}"]
          \arrow[rr, equal]
          &&
          \xusebox{monoid-white}
        \end{tikzcd}
        \quad
        \begin{tikzcd}[math mode=false, sep=0pt, row sep=-15pt]
          &
          \xusebox{white-unit-triangle-2}
          \arrow[rd, Rightarrow, "\textcolor{red}{$\psi_I$}"]
          &
          \\
          \xusebox{white-unit-triangle-1}
          \arrow[ru, Rightarrow, "\textcolor{red}{$\varphi_I$}", start anchor={[xshift=\cobgap, yshift=\cobgap]center}, shorten >=10pt]
          \arrow[rr, equal]
          &&
          \xusebox{unit-white}
        \end{tikzcd}
        \quad
        \begin{tikzcd}[math mode=false, sep=0pt, row sep=-5pt]
          \xusebox{black-comonoid-triangle-1}
          \arrow[rr, equal]
          \arrow[rd, Rightarrow, "\textcolor{red}{$\eta_\otimes$}"', shorten=-5pt]
          &&
          \xusebox{comonoid-black}
          \\
          &
          \xusebox{black-comonoid-triangle-2}
          \arrow[ru, Rightarrow, "\textcolor{red}{$\varepsilon_\otimes$}"', shorten <=-5pt, start anchor={east}]
          &
        \end{tikzcd}
        \quad
        \begin{tikzcd}[math mode=false, sep=0pt, row sep=-15pt]
          \xusebox{black-counit-triangle-1}
          \arrow[rr, equal]
          \arrow[rd, Rightarrow, "\textcolor{red}{$\varphi_I$}"', start anchor={[xshift=\cobgap, yshift=-\cobgap]center}, shorten >=10pt]
          &&
          \xusebox{counit-black}
          \\
          &
          \xusebox{black-counit-triangle-2}
          \arrow[ru, Rightarrow, "\textcolor{red}{$\psi_I$}"']
          &
        \end{tikzcd}
      \end{equation*}
  \end{itemize}
\end{example}

\begin{lemma}
  In the free monoidal bicategory on $\ra{\mathcal{M}}$, $(\cdot,
  \xusebox{tinycomonoid-black}, \xusebox{tinycounit-black})$ can be given a
  canonical pseudocomonoid structure, by transporting the pseudomonoid $(\cdot,
  \xusebox{tinymonoid-white}, \xusebox{tinyunit-white})$ across the
  adjunctions.
\end{lemma}

\begin{lemma}
  Interpretations of $\ra{\mathcal{M}}$ in $\Prof$ correspond to Cauchy
  complete monoidal categories.
\end{lemma}

\subsection{Braiding and symmetry}

\begin{definition}\label[definition]{braided}
  \begin{xlrbox}{braid-coherence-1}
    \begin{tikzpicture}[baseline=(current bounding box.center)]
      \node[mult] (m) {};
      \node[mult, anchor=top] (m') at (m.leftleg) {};
      \node[identity, anchor=top] at (m.rightleg) {};
      \node[draw=red, dashed, fit={(m.top) (m'.leftleg) (m.rightleg)}]{};
      \node[draw=blue, dashed, fit={(m'.top) (m'.leftleg) (m'.rightleg)}]{};
    \end{tikzpicture}
  \end{xlrbox}

  \begin{xlrbox}{braid-coherence-2}
    \begin{tikzpicture}[baseline=(current bounding box.center)]
      \node[mult] (m) {};
      \node[mult, anchor=top] (m') at (m.rightleg) {};
      \node[identity, anchor=top] at (m.leftleg) {};
      \node[draw=red, dashed, fit={(m.top) (m.leftleg) (m.rightleg)}]{};
    \end{tikzpicture}
  \end{xlrbox}

  \begin{xlrbox}{braid-coherence-3}
    \begin{tikzpicture}[baseline=(current bounding box.center)]
      \node[mult] (m) {};
      \node[braid, anchor=toprightleg] (b) at (m.rightleg) {};
      \node[mult, anchor=top] at (b.bottomrightleg) {};
      \node[identity, anchor=top] at (b.bottomleftleg) {};
    \end{tikzpicture}
  \end{xlrbox}

  \begin{xlrbox}{braid-coherence-4}
    \begin{tikzpicture}[baseline=(current bounding box.center)]
      \node[mult, span=1.5] (m) {};
      \node[mult, anchor=top] (m') at (m.leftleg) {};
      \node[braid, anchor=topleftleg] (b) at (m'.rightleg) {};
      \node[braid, anchor=toprightleg] at (b.bottomleftleg) {};
      \node[identity, anchor=top] at (b.bottomrightleg) {};
      \node[identity, anchor=top] at (m'.leftleg) {};
      \node[identity, anchor=top] at (m.rightleg) {};
      \node[draw=red, dashed, fit={(m.top) (m'.leftleg) (m.rightleg)}]{};
    \end{tikzpicture}
  \end{xlrbox}

  \begin{xlrbox}{braid-coherence-5}
    \begin{tikzpicture}[baseline=(current bounding box.center)]
      \node[mult] (m) {};
      \node[mult, anchor=top] (m') at (m.leftleg) {};
      \idstack[anchor=m.rightleg]{2}
      \node[braid, anchor=topleftleg] at (m'.leftleg) {};
      \node[draw=blue, dashed, fit={(m.top) (m'.leftleg) (m.rightleg)}]{};
    \end{tikzpicture}
  \end{xlrbox}

  \begin{xlrbox}{braid-coherence-6}
    \begin{tikzpicture}[baseline=(current bounding box.center)]
      \node[mult, span=1.5] (m) {};
      \node[mult, anchor=top] (m') at (m.rightleg) {};
      \idstack[anchor=m.leftleg]{1}
      \idstack[anchor=m'.rightleg]{1}
      \node[braid, anchor=toprightleg] at (m'.leftleg) {};
      \node[draw=blue, dashed, fit={(m'.top) (m'.leftleg) (m'.rightleg)}]{};
    \end{tikzpicture}
  \end{xlrbox}

  \begin{xlrbox}{braid-coherence-7}
    \begin{tikzpicture}[baseline=(current bounding box.center)]
      \node[mult, span=1.5] (m) {};
      \node[mult, anchor=top] (m') at (m.rightleg) {};
      \node[braid, anchor=topleftleg] (b) at (m'.leftleg) {};
      \node[braid, anchor=toprightleg] at (b.bottomleftleg) {};
      \node[identity, anchor=top] at (b.bottomrightleg) {};
      \idstack[anchor=m.leftleg]{2}
    \end{tikzpicture}
  \end{xlrbox}

  The \emph{braided pseudomonoid presentation} $\mathcal{B}$ is given by the data of $\mathcal{M}$, and an
  additional $2$-morphism specifying that the pseudomonoid is commutative:
  \begin{equation*}
    \xusebox{monoid-white}
    \overset{\sigma}{\cong}
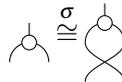
    \begin{tikzpicture}[baseline=(current bounding box.center)]
      \node[mult] (m) {};
      \node[braid, anchor=toprightleg] at (m.rightleg) {};
    \end{tikzpicture}
  \end{equation*}
  and coherence equation (hexagon):

  \begin{equation*}
    \begin{tikzcd}[math mode=false, sep=-45pt, column sep=tiny]
      &
      \xusebox{braid-coherence-2}
      \arrow[rr, Rightarrow, "\textcolor{red}{$\sigma$}"]
      &&
      \xusebox{braid-coherence-3}
      \arrow[rr, Rightarrow]
      &&
      \xusebox{braid-coherence-4}
      \arrow[rd, Rightarrow, "\textcolor{red}{$\alpha$}"]
      &
      \\
      \xusebox{braid-coherence-1}
      \arrow[ru, Rightarrow, "\textcolor{red}{$\alpha$}"]
      \arrow[rrd, Rightarrow, "\textcolor{blue}{$\sigma$}"']
      &&&&&&
      \xusebox{braid-coherence-7}
      \\
      &&
      \xusebox{braid-coherence-5}
      \arrow[rr, Rightarrow, "\textcolor{blue}{$\alpha$}"']
      &&
      \xusebox{braid-coherence-6}
      \arrow[rru, Rightarrow, "\textcolor{blue}{$\sigma$}"']
      &&
    \end{tikzcd}
  \end{equation*}
\end{definition}

\begin{lemma}
  Interpretations of $\ra{\mathcal{B}}$ in $\Prof$ correspond to Cauchy
  complete braided monoidal categories.
\end{lemma}

Likewise, \enquote{braided} can be promoted to \enquote{balanced} by adding a
compatible twist in the presentation.

\begin{definition}
  \begin{xlrbox}{braid-twist-coherence-1}
    \begin{tikzpicture}[baseline=(current bounding box.center)]
      \node[mult] (m) {};
      \node[draw=red, dashed, fit={(m.top)}]{};
      \node[draw=blue, dashed, fit={(m.top) (m.leftleg) (m.rightleg)}]{};
    \end{tikzpicture}
  \end{xlrbox}

  \begin{xlrbox}{braid-twist-coherence-2}
    \begin{tikzpicture}[baseline=(current bounding box.center)]
      \node[mult] (m) {};
      \node[braid, anchor=toprightleg] (b) at (m.rightleg) {};
      \node[braid, anchor=toprightleg] at (b.bottomrightleg) {};
    \end{tikzpicture}
  \end{xlrbox}

  \begin{xlrbox}{braid-twist-coherence-3}
    \begin{tikzpicture}[baseline=(current bounding box.center)]
      \node[mult] (m) {};
      \node[draw=red, dashed, fit={(m.leftleg)}]{};
      \node[draw=blue, dashed, fit={(m.rightleg)}]{};
    \end{tikzpicture}
  \end{xlrbox}

  \begin{xlrbox}{twist-coherence-1}
    \begin{tikzpicture}[baseline=(current bounding box.center)]
      \node[unit] (m) {};
      \node[draw=red, dashed, fit={(m.top)}]{};
    \end{tikzpicture}
  \end{xlrbox}

  The \emph{balanced pseudomonoid presentation} $\mathcal{L}$ is given by the data of
  $\mathcal{B}$, and additionally a $2$-endomorphism specifying a compatible
  twist:
  \begin{equation*}
    \xusebox{id}
    \overset{\theta}{\cong}
    \xusebox{id}
  \end{equation*}
  and equations:
  
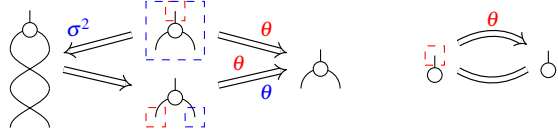
\begin{equation*}
    \begin{tikzcd}[math mode=false, row sep=-27.5pt]
      &
      \xusebox{braid-twist-coherence-1}
      \arrow[rd, Rightarrow, "\textcolor{red}{$\theta$}"]
      \arrow[dl, Rightarrow, "\textcolor{blue}{$\sigma^2$}"']
      &
      \\
      \xusebox{braid-twist-coherence-2}
      \arrow[rd, Rightarrow]
      &&
      \xusebox{monoid-white}
      \\
      &
      \xusebox{braid-twist-coherence-3}
      \arrow[ru, Rightarrow,  "\textcolor{red}{$\theta$}", "\textcolor{blue}{$\theta$}"']
      &
    \end{tikzcd}
    \qquad
    \begin{tikzcd}[math mode=false]
      \xusebox{twist-coherence-1}
      \arrow[r, equals, bend right]
      \arrow[r, Rightarrow, bend left, "\textcolor{red}{$\theta$}"]
      &
      \begin{tikzpicture}[baseline=(current bounding box.center)]
        \node[unit] (m) {};
      \end{tikzpicture}
    \end{tikzcd}
  \end{equation*}
\end{definition}

\begin{remark}\label[remark]{symmetry-balanced}
  The coherence equation of \Cref{braided} is redundant in the presence of a
  twist. Conversely, the symmetric pseudomonoid presentation is equivalent to
  the balanced pseudomonoid presentation with a trivial twist.
\end{remark}

Subsequently, we shall examine a variety of presentations, representing
different types of monoidal categories, which extend $\mathcal{M}$: in each
case, their braided (resp.\ balanced) variant is obtained by considering the
corresponding extension with respect to $\mathcal{B}$ (resp.\
$\mathcal{L}$) instead.

\subsection{Autonomous categories}

A monoidal category is autonomous when every object has a left and a right dual. Here we recall how they can be defined via a presentation following \textcite{bartlettModularCategoriesRepresentations2015}.

\begin{definition}\label[definition]{autonomous}
  The \emph{autonomous pseudomonoid presentation} $\mathcal{A}$ is given by the
  data of $\ra{\mathcal{M}}$, and additionally inverses for the following composite
  $2$-morphisms:
  \begin{align*}
    \gamma_L &\coloneqq
    \begin{tikzpicture}[baseline=(current bounding box.center)]
      \node[comult, dot=black] (c) {};
      \node[mult, anchor=rightleg] (m) at (c.leftleg) {};
      \node[identity, anchor=top] at (m.leftleg) {};
      \node[identity, anchor=bottom] (i) at (c.rightleg) {};
      \node[draw=red, dashed, fit={(m.top) (i.top)}]{};
    \end{tikzpicture}
    \xRightarrow{\textcolor{red}{\eta_\otimes}}
    \begin{tikzpicture}[baseline=(current bounding box.center)]
      \node[comult, dot=black] (c) {};
      \node[mult, anchor=rightleg] (m) at (c.leftleg) {};
      \node[identity, anchor=top] at (m.leftleg) {};
      \node[identity, anchor=bottom] at (c.rightleg) {};
      \node[mult, anchor=leftleg, span=1.5] (m') at (m.top) {};
      \node[comult, dot=black, anchor=bottom] (c') at (m'.top) {};
      \node[draw=red, dashed, fit={(m'.north) (m.leftleg) (m'.rightleg)}]{};
    \end{tikzpicture}
    \overset{\textcolor{red}{\alpha}}{\cong}
    \begin{tikzpicture}[baseline=(current bounding box.center)]
      \node[mult] (m) {};
      \node[comult, dot=black, anchor=bottom] (c) at (m.top) {};
      \node[mult, anchor=top] (m') at (m.rightleg) {};
      \node[comult, dot=black, anchor=leftleg] (c') at (m'.leftleg) {};
      \idstack[anchor=m.leftleg]{2}
      \node[draw=red, dashed, fit={(m'.top) (m'.leftleg) (m'.rightleg) (c'.bottom)}]{};
    \end{tikzpicture}
    \xRightarrow{\textcolor{red}{\varepsilon_\otimes}}
    \xusebox{eta-prime-target}
    &
    \gamma_R &\coloneqq
    \begin{tikzpicture}[baseline=(current bounding box.center)]
      \node[comult, dot=black] (c) {};
      \node[mult, anchor=leftleg] (m) at (c.rightleg) {};
      \node[identity, anchor=top] at (m.rightleg) {};
      \node[identity, anchor=bottom] (i) at (c.leftleg) {};
      \node[draw=red, dashed, fit={(m.top) (i.top)}]{};
    \end{tikzpicture}
    \xRightarrow{\textcolor{red}{\eta_\otimes}}
    \begin{tikzpicture}[baseline=(current bounding box.center)]
      \node[comult, dot=black] (c) {};
      \node[mult, anchor=leftleg] (m) at (c.rightleg) {};
      \node[identity, anchor=top] at (m.rightleg) {};
      \node[identity, anchor=bottom] (i) at (c.leftleg) {};
      \node[mult, anchor=rightleg, span=1.5] (m') at (m.top) {};
      \node[comult, dot=black, anchor=bottom] (c') at (m'.top) {};
      \node[draw=red, dashed, fit={(m'.north) (m.rightleg) (m'.leftleg)}]{};
    \end{tikzpicture}
    \overset{\textcolor{red}{\alpha^{-1}}}{\cong}
    \begin{tikzpicture}[baseline=(current bounding box.center)]
      \node[mult] (m) {};
      \node[comult, dot=black, anchor=bottom] (c) at (m.top) {};
      \node[mult, anchor=top] (m') at (m.leftleg) {};
      \node[comult, dot=black, anchor=leftleg] (c') at (m'.leftleg) {};
      \idstack[anchor=m.rightleg]{2}
      \node[draw=red, dashed, fit={(m'.top) (m'.leftleg) (m'.rightleg) (c'.bottom)}]{};
    \end{tikzpicture}
    \xRightarrow{\textcolor{red}{\varepsilon_\otimes}}
    \xusebox{eta-prime-target}
  \end{align*}
\end{definition}

\begin{lemma}[{\autocite[Proposition~4.8]{bartlettModularCategoriesRepresentations2015}}]
  Interpretations of $\mathcal{A}$ in $\Prof$ correspond to Cauchy
  complete autonomous categories.
\end{lemma}

\subsection{Internal string diagrams}

\begin{figure}[!t]
\scalecobordisms{1}
\begin{calign}
\nonumber
\begin{tikzpicture}[baseline=(current bounding box.center)]
    \node[Cyl, top, bot, tall] (i) {};
    \node[Cyl, top, bot, tall, xshift=\cobgap+\cobwidth] (j) {};
    \begin{scope}[internal string scope]
        \node (a) at ([yshift=\toff] i.top) [above] {$A$};
        \node (c) at ([yshift=-\boff] i.bot) [below] {$C$};
        \draw (c.north) to [out=up, in=down] (a.south);
        \node [tiny label] at (i.center) {$f$};
        \node (b) at ([yshift=\toff] j.top) [above] {$B$};
        \node (d) at ([yshift=-\boff] j.bot) [below] {$D$};
        \draw (d.north) to [out=up, in=down] (b.south);
        \node [tiny label] at (j.center) {$g$};
    \end{scope}
  \end{tikzpicture}
  \xRightarrow{\eta_\otimes}
  \begin{tikzpicture}[baseline=(current bounding box.center)]
    \node[Pants, bot, belt scale=1.5] (pants) {};
    \node[Copants, top, bot, anchor=belt, belt scale=1.5] (copants) at (pants.belt) {};
    \begin{scope}[internal string scope]
        \node (a) at ([yshift=\toff] copants.leftleg) [above] {$A$};
        \node (c) at ([yshift=-\boff] pants.leftleg) [below] {$C$};
        \node [tiny label, minimum width=0.2cm] (f) at ([xshift=-0.25\cobwidth]pants.belt) {$f$};
        \draw (c.north) to [out=up, in=down] (f.center) to [out=up, in=down] (a.south);
        \node (b) at ([yshift=\toff] copants.rightleg) [above] {$B$};
        \node (d) at ([yshift=-\boff] pants.rightleg) [below] {$D$};
        \node [tiny label, minimum width=0.2cm] (g) at ([xshift=0.25\cobwidth]pants.belt) {$g$};
        \draw (d.north) to [out=up, in=down] (g.center) to [out=up, in=down] (b.south);
    \end{scope}
  \end{tikzpicture}
&
    \begin{tikzpicture}[baseline=(current bounding box.center)]
    \node[Pants, bot, top] (B) at (0,0) {};
    \node[Pants, bot, anchor=belt] (A) at (B.leftleg) {};
    \node[SwishL, bot, anchor=top] (C) at (B.rightleg) {};
    \begin{scope}[internal string scope]
        \node (i) at (B.belt) [above=\toff] {$A$};
        \node (j) at (A.leftleg) [below=\boff] {$B$};
        \node (k) at (A.rightleg) [below=\boff] {$C$};
        \node (l) at (C.bot) [below=\boff] {$D$};
        \node [tiny label] (f) at (0,0.13) {$f$};
        \node [tiny label] (g) at (A.center) {$g$};
        \node [tiny label] (h) at (C.center) {$h$};
        \draw (f.center) to (i.south);
        \draw (j.north) to [out=up, in=-135] (g.center);
        \draw (k.north) to [out=up, in=-35] (g.center);
        \draw (g.center) to [out=90, in=-135] node [left=-4pt] {} (f.center);
        \draw (l.north)
            to [out=90, in=-80] (h.center)
            to [out=up, in=down, out looseness=0.7]
                (B.rightleg)
            to [out=up, in=-45]
                node [right=-4pt, pos=0.11] {}
  (f.center);
    \end{scope}
  \end{tikzpicture}
  \overset{\alpha}{\cong}
  \begin{tikzpicture}[baseline=(current bounding box.center)]
    \node[Pants, bot, top] (B) at (0,0) {};
    \node[Pants, bot, anchor=belt] (A) at (B.rightleg) {};
    \node[SwishR, bot, anchor=top] (C) at (B.leftleg) {};
    \begin{scope}[internal string scope]
        \node (i) at (B.belt) [above=\toff] {$A$};
        \node (j) at (C.bot) [below=\boff] {$B$};
        \node (k) at (A.leftleg) [below=\boff] {$C$};
        \node (l) at (A.rightleg) [below=\boff] {$D$};
        \node [tiny label] (f) at (0.05\cobwidth,0.25\cobheight) {$f$};
        \node [tiny label] (g)  at (-0.25\cobwidth,-0.1\cobheight) {$g$};
        \draw (j.north)
            to [out=up, in=-130] (g.center);
        \draw (k.north)
            to [out=up, in=down] (B-rightleg.in-leftthird)
            to [out=up, in=-40] (g.center);
        \draw (l.north) to [out=up, in=down] (B-rightleg.in-rightthird)
            to [out=up, in=-60] (f.center);
        \draw (f.center) to [out=90, in=-90, looseness=2] (i.south);
        \draw (f.center) to [in=90, out=-120] (g.center);
        \node [tiny label] (h) at (0.8\cobwidth,-0.25\cobheight) {$h$};
    \end{scope}
  \end{tikzpicture}
&
  \begin{tikzpicture}[baseline=(current bounding box.center)]
    \node[Cyl, top, bot] (i) {};
    \node[Cyl, bot, anchor=top] (j) at (i.bottom) {};
    \begin{scope}[internal string scope]
        \node (x) at ([yshift=\toff] i.top) [above] {$A$};
        \node (y) at ([yshift=-\boff] j.bot) [below] {$B$};
        \draw (y.north) to [out=up, in=down] (x.south);
        \node [tiny label] at (i.center) {$f$};
    \end{scope}
  \end{tikzpicture}
  \cong
  \begin{tikzpicture}[baseline=(current bounding box.center)]
    \node[Cyl, top, bot] (i) {};
    \node[Cyl, bot, anchor=top] (j) at (i.bottom) {};
    \begin{scope}[internal string scope]
        \node (x) at ([yshift=\toff] i.top) [above] {$A$};
        \node (y) at ([yshift=-\boff] j.bot) [below] {$B$};
        \draw (y.north) to [out=up, in=down] (x.south);
        \node [tiny label] at (j.center) {$f$};
    \end{scope}
  \end{tikzpicture}
  \\\nonumber
  \mathrm{(a)} & \mathrm{(b)} & \mathrm{(c)}
  \end{calign}

  \caption{Examples of the internal string diagram formalism.}
  \label{fig:internalstring}
\end{figure}
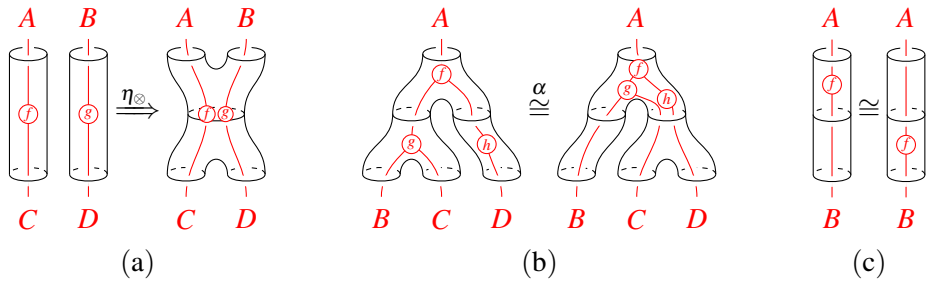

One special aspect of $\Prof$ is that it admits a calculus called
the \emph{internal string diagram} construction, when we consider presentations
which extend $\ra{\mathcal{M}}$. Informally speaking, the strings we use can
be inflated into tubes, containing a volume in which the standard graphical
calculus for monoidal categories operates, and $2$-morphisms in $\Prof$
correspond to rewrites of these tubes which act on the internal strings.

Some examples of the formalism are shown in Figure~\ref{fig:internalstring}. A feature of the formalism is that the internal strings must be read in the opposite direction to the ambient profunctors. Since our profunctor convention is bottom-to-top, the convention for internal strings is top-to-bottom.

\begin{example}
Figure~\ref{fig:internalstring}(a) illustrates the action of
$\tinymorphisms{
  \node[identity] (i) {};
  \node[identity, anchor=top] at (i.bottom) {};
  \node[identity, xshift=\cobgap+\cobwidth] (i') {};
  \node[identity, anchor=top] at (i'.bottom) {};
} \xRightarrow{\eta_\otimes}
\tinymorphisms{
  \node[mult] (m) {};
  \node[comult, dot=black, anchor=bottom] at (m.top) {};
}$ on internal strings. As a function of sets it maps $\cat{C}(A, C) \times
\cat{C}(B, D) \to \cat{C}(A \otimes B, C \otimes D)$, which sends $(f, g) \mapsto f
\otimes g$.

In Figure~\ref{fig:internalstring}(b) we show the action of the associator
$\tinymorphisms{
  \node[mult] (m) {};
  \node[mult, anchor=leftleg] (m') at (m.top) {};
  \node[identity, anchor=top] (i) at (m'.rightleg) {};
} \overset{\alpha}{\cong}
\tinymorphisms{
  \node[mult] (m) {};
  \node[mult, anchor=rightleg] (m') at (m.top) {};
  \node[identity, anchor=top] (i) at (m'.leftleg) {};
}$. As a function of sets, this natural transformation of profunctors has type
 $\cat{C}(A, (B
\otimes C) \otimes D) \to \cat{C}(A, B \otimes (C \otimes D))$, and acts by post-composition of $(B \otimes C) \otimes D \xrightarrow{\alpha_{B, C, D}} B \otimes
(C \otimes D)$.
\end{example}

The definition of profunctor composition for two profunctors $F\colon \op{\mathcal{B}} \times \cat{A} \to \Set$ and $G\colon \op{\mathcal{C}} \times \cat{B} \to \Set$ is
$G \diamond F \coloneqq \int^b G(-, b) \times F(b, =)$, where $\int$ denotes a coend \autocite[Equation~5.1]{loregianCoendCalculus2019}. This is interpreted in $\Set$ as a disjoint union quotiented by the least equivalence
relation generated by $(f \cdot g, h) \sim (f, g \cdot h)$.
From an internal string diagram perspective, this precisely says that morphisms can \enquote{move freely} through boundary circles. We illustrate this in Figure~\ref{fig:internalstring}(c).

A formal treatment of internal string diagrams is given in \textcite[\S~4]{bartlettModularCategoriesRepresentations2015}, which establishes that it is a sound calculus for reasoning about  $\ra{\mathcal{M}}$.

\section{The traced pseudomonoid presentation}\label{sec:traced-presentation}


In this section, we introduce our main contribution: the algebraic gadget in
$\Prof$ which admits traced monoidal (Cauchy complete) categories as
interpretations. This offers an \emph{external} perspective on traced monoidal
categories, akin to how monoidal categories can be viewed externally as
pseudomonoids internal to $\Cat$, versus a category equipped with
\emph{internal} structure.

\subsection{Traced monoidal categories}

In our framework, we seek to capture the standard notion of traced monoidal
category in terms of a presentation.

\begin{xlrbox}{tr-source}
  \begin{tikzpicture}[baseline=(current bounding box.center), decoration={
          markings, mark=at position 0.5 with {\arrowreversed[black]{Stealth[length=1mm]}}
        }]
    \node[comult, dot=black] (c) {};
    \node[mult, anchor=top] (m) at (c.bottom) {};
    \node[2Dcup, anchor=leftleg] (cup) at (m.rightleg) {};
    \node[2Dcap, anchor=leftleg] (cap) at (c.rightleg) {};
    \idstack[anchor=m.leftleg, direction=down]{1}
    \idstack[anchor=c.leftleg, direction=up]{1}
    \draw[postaction={decorate}] (cup.rightleg) -- (cap.rightleg);
  \end{tikzpicture}
\end{xlrbox}

\begin{xlrbox}{rtr-source}
  \begin{tikzpicture}[baseline=(current bounding box.center)]
    \node[mult] (m) {};
    \node[mult, dot=black, anchor=leftleg, span=1.5] (m') at (m.top) {};
    \node[twfrobcup, anchor=leftleg] (cup) at (m.rightleg) {};
    \node[identity, anchor=top] at (m'.rightleg) {};
    \idstack[anchor=m.leftleg]{2}
  \end{tikzpicture}
\end{xlrbox}

\begin{xlrbox}{ltr-source}
  \begin{tikzpicture}[baseline=(current bounding box.center)]
    \node[comult] (c) {};
    \node[twfrobcap, anchor=rightleg] (cap) at (c.leftleg) {};
    \node[comult, dot=black, anchor=rightleg, span=1.5] (c') at (c.bottom) {};
    \idstack[anchor=cap.leftleg]{1}
    \idstack[anchor=c.rightleg, direction=up]{2}
  \end{tikzpicture}
\end{xlrbox}

\begin{xlrbox}{ntr-source}
  \begin{tikzpicture}[baseline=(current bounding box.center)]
    \node[comult, dot=black] (c) {};
    \node[mult, anchor=top] (m) at (c.bottom) {};
    \node[frobcap, autonomous, anchor=leftleg] at (c.rightleg) {};
    \node[frobcup, autonomous, anchor=leftleg] (c') at (m.rightleg) {};
    \idstack[anchor=c'.rightleg, direction=up] {2}
    \idstack[anchor=m.leftleg, direction=down] {1}
    \idstack[anchor=c.leftleg, direction=up] {1}
  \end{tikzpicture}
\end{xlrbox}

\begin{definition}\label[definition]{traced-pseudomonoid}
  \begin{xlrbox}{tr-van-i-1}
    \begin{tikzpicture}[baseline=(current bounding box.center)]
      \node[comult, dot=black] (c) {};
      \node[mult, anchor=top] (m) at (c.bottom) {};
      \node[counit, dot=black, anchor=bottom] (cu) at (c.rightleg) {};
      \node[unit, anchor=top] (u) at (m.rightleg) {};
      \idstack[anchor=m.leftleg, direction=down]{1}
      \idstack[anchor=c.leftleg, direction=up]{1}
      \node[draw=blue, dashed, fit={(u) (m.north) (m.leftleg)}]{};
      \node[draw=blue, dashed, fit={(cu) (c.south) (c.leftleg)}]{};
    \end{tikzpicture}
  \end{xlrbox}

  \begin{xlrbox}{tr-van-i-2}
    \begin{tikzpicture}[baseline=(current bounding box.center)]
      \node[comult, dot=black] (c) {};
      \node[mult, anchor=top] (m) at (c.bottom) {};
      \node[2Dcap, anchor=leftleg] (cap) at (c.rightleg) {};
      \node[2Dcup, anchor=leftleg] (cup) at (m.rightleg) {};
      \node[counit, down, anchor=bottom] (cu) at (cup.rightleg) {};
      \node[unit, dot=black, down, anchor=top] (u) at (cap.rightleg) {};
      \idstack[anchor=m.leftleg, direction=down]{1}
      \idstack[anchor=c.leftleg, direction=up]{1}
      \node[draw=red, dashed, fit={(cu) (u) (cu.bottom) (u.top)}]{};
    \end{tikzpicture}
  \end{xlrbox}

  \begin{xlrbox}{tr-van-tensor-1}
    \begin{tikzpicture}[baseline=(current bounding box.center), decoration={
        markings, mark=at position 0.5 with {\arrowreversed[black]{Stealth[length=1mm]}}
      }]
      \node[comult, dot=black] (c) {};
      \node[mult, anchor=top] (m) at (c.bottom) {};
      \node[2Dcap, anchor=leftleg] (cap) at (c.rightleg) {};
      \node[2Dcup, anchor=leftleg] (cup) at (m.rightleg) {};
      \node[comult, dot=black, anchor=bottom] (c') at (c.leftleg) {};
      \node[mult, anchor=top] (m') at (m.leftleg) {};
      \node[2Dcap, anchor=leftleg, span=2] (cap') at (c'.rightleg) {};
      \node[2Dcup, anchor=leftleg, span=2] (cup') at (m'.rightleg) {};
      \idstack[anchor=m'.leftleg, direction=down]{1}
      \idstack[anchor=c'.leftleg, direction=up]{1}
      \draw[postaction={decorate}] (cup.rightleg) -- (cap.rightleg);
      \draw[postaction={decorate}] (cup'.rightleg) -- (cap'.rightleg);
      \node[draw=red, dashed, fit={(c.leftleg) (m.leftleg) (cap.north) (cup.south) (cup.rightleg)}]{};
      \node[draw=blue, dashed, fit={(c.south) (c.rightleg) (c'.leftleg)}]{};
      \node[draw=blue, dashed, fit={(m.north) (m.rightleg) (m'.leftleg)}]{};
    \end{tikzpicture}
  \end{xlrbox}

  \begin{xlrbox}{tr-van-tensor-2}
    \begin{tikzpicture}[baseline=(current bounding box.center), decoration={
        markings, mark=at position 0.5 with {\arrowreversed[black]{Stealth[length=1mm]}}
      }]
      \node[comult, dot=black] (c) {};
      \node[mult, anchor=top] (m) at (c.bottom) {};
      \node[comult, dot=black, anchor=bottom] (c') at (c.rightleg) {};
      \node[mult, anchor=top] (m') at (m.rightleg) {};
      \node[identity, anchor=bottom] (it) at (c'.leftleg) {};
      \node[identity, anchor=top] (ib) at (m'.leftleg) {};
      \node[2Dcap, anchor=leftleg] (cap) at (c'.rightleg) {};
      \node[2Dcup, anchor=leftleg] (cup) at (m'.rightleg) {};
      \node[2Dcap, anchor=leftleg, span=2.5] (cap') at (it.top) {};
      \node[2Dcup, anchor=leftleg, span=2.5] (cup') at (ib.bottom) {};
      \idstack[anchor=m.leftleg, direction=down]{4}
      \idstack[anchor=c.leftleg, direction=up]{4}
      \draw[postaction={decorate}] (cup.rightleg) -- (cap.rightleg);
      \draw[postaction={decorate}] (cup'.rightleg) -- (cap'.rightleg);
    \end{tikzpicture}
  \end{xlrbox}

  \begin{xlrbox}{tr-van-tensor-3}
    \begin{tikzpicture}[baseline=(current bounding box.center)]
      \node[comult, dot=black] (c) {};
      \node[mult, anchor=top] (m) at (c.bottom) {};
      \node[2Dcup, anchor=leftleg] (cup) at (m.rightleg) {};
      \node[2Dcap, anchor=leftleg] (cap) at (c.rightleg) {};
      \node[mult, down, dot=black, anchor=top] (c') at (cap.rightleg) {};
      \node[comult, down, anchor=bottom] (m') at (cup.rightleg) {};
      \idstack[anchor=m.leftleg, direction=down]{1}
      \idstack[anchor=c.leftleg, direction=up]{1}
      \node[draw=red, dashed, fit={(c'.rightleg) (m'.leftleg) (c'.north) (m'.south)}]{};
    \end{tikzpicture}
  \end{xlrbox}

  \begin{xlrbox}{tr-sup-1}
    \begin{tikzpicture}[baseline=(current bounding box.center), decoration={
        markings, mark=at position 0.5 with {\arrowreversed[black]{Stealth[length=1mm]}}
      }]
      \node[comult, dot=black] (c) {};
      \node[mult, anchor=top] (m) at (c.bottom) {};
      \node[2Dcup, anchor=leftleg] (cup) at (m.rightleg) {};
      \node[2Dcap, anchor=leftleg] (cap) at (c.rightleg) {};
      \idstack[anchor=m.leftleg, direction=down]{1}
      \idstack[anchor=c.leftleg, direction=up]{1}
      \node[identity, xshift=-\cobwidth-\cobgap, anchor=top] (i) at (id1.top) {};
      \idstack[anchor=i.bottom, direction=down]{3}
      \draw[postaction={decorate}] (cup.rightleg) -- (cap.rightleg);
      \node[draw=red, dashed, fit={(cap.rightleg) (cap.north) (cup.south) (c.leftleg) (m.leftleg)}]{};
      \node[draw=blue, dashed, fit={(m.top) (id2.top)}]{};
    \end{tikzpicture}
  \end{xlrbox}

  \begin{xlrbox}{tr-sup-2}
    \begin{tikzpicture}[baseline=(current bounding box.center), decoration={
        markings, mark=at position 0.5 with {\arrowreversed[black]{Stealth[length=1mm]}}
      }]
      \node[comult, dot=black] (c) {};
      \node[mult, anchor=top] (m) at (c.bottom) {};
      \node[comult, dot=black, anchor=bottom] (c') at (c.rightleg) {};
      \node[mult, anchor=top] (m') at (m.rightleg) {};
      \node[2Dcap, anchor=leftleg] (cap) at (c'.rightleg) {};
      \node[2Dcup, anchor=leftleg] (cup) at (m'.rightleg) {};
      \idstack[anchor=c'.leftleg, direction=up]{1}
      \idstack[anchor=m'.leftleg, direction=down]{1}
      \idstack[anchor=m.leftleg, direction=down]{2}
      \idstack[anchor=c.leftleg, direction=up]{2}
      \draw[postaction={decorate}] (cup.rightleg) -- (cap.rightleg);
      \node[draw=red, dashed, fit={(c.south) (c.leftleg) (c'.rightleg)}]{};
      \node[draw=red, dashed, fit={(m.north) (m.leftleg) (m'.rightleg)}]{};
    \end{tikzpicture}
  \end{xlrbox}

  \begin{xlrbox}{tr-sup-3}
    \begin{tikzpicture}[baseline=(current bounding box.center), decoration={
        markings, mark=at position 0.5 with {\arrowreversed[black]{Stealth[length=1mm]}}
      }]
      \node[comult, dot=black] (c) {};
      \node[mult, anchor=top] (m) at (c.bottom) {};
      \node[2Dcap, anchor=leftleg] (cap) at (c.rightleg) {};
      \node[2Dcup, anchor=leftleg] (cup) at (m.rightleg) {};
      \idstack[anchor=c.leftleg, direction=up]{1}
      \node[comult, dot=black, anchor=bottom] (c') at (id1.top) {};
      \idstack[anchor=m.leftleg, direction=down]{1}
      \node[mult, anchor=top] (m') at (id1.bottom) {};
      \draw[postaction={decorate}] (cup.rightleg) -- (cap.rightleg);
      \node[draw=red, dashed, fit={(cap.rightleg) (cap.north) (cup.south) (c.leftleg) (m.leftleg)}]{};
    \end{tikzpicture}
  \end{xlrbox}

  The \emph{traced pseudomonoid presentation} $\mathcal{T}$ is given by the data of $\ra{\mathcal{M}}$, and additionally:
  \begin{itemize}
    \item a generating $2$-morphism:
      \begin{equation}\label{eq:tr}\tag{$\Tr$}
        \xusebox{tr-source}
        \xRightarrow{\Tr}
        \xusebox{veryverytallid}
      \end{equation}
    \item equations:

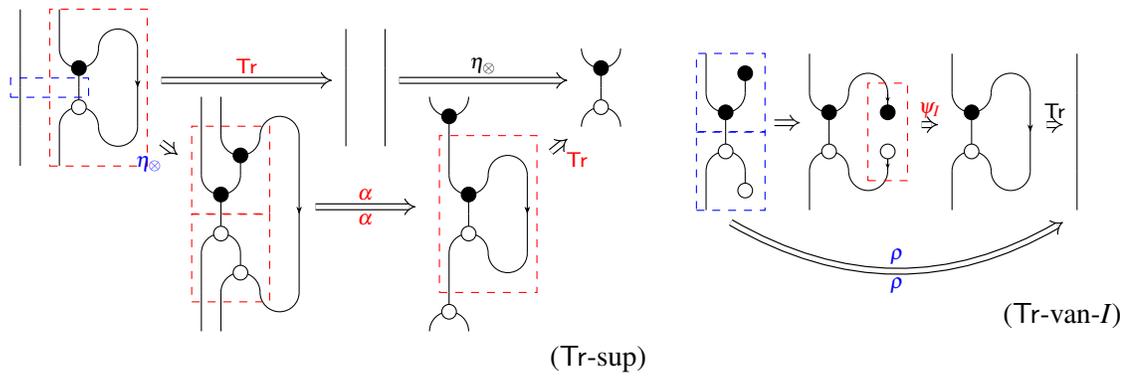
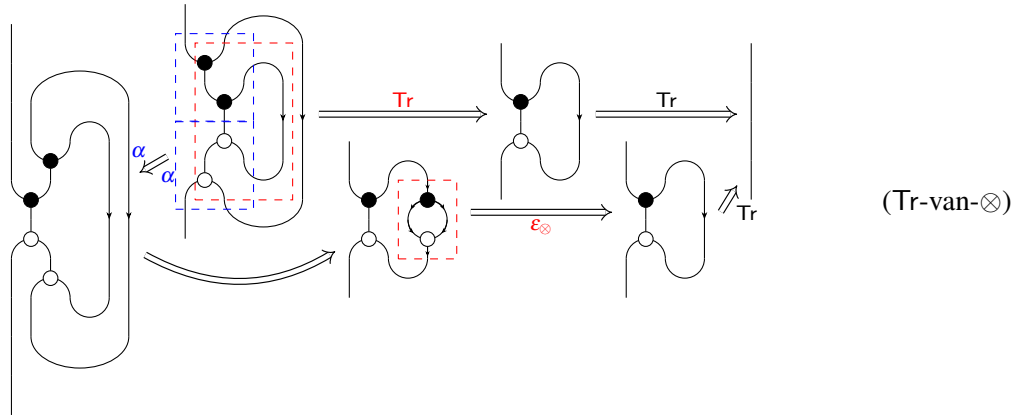
    \begin{minipage}{0.55\textwidth}
      \begin{equation}
      \begin{tikzcd}[math mode=false, sep=tiny, row sep=-35pt, ampersand replacement=\&]
        \xusebox{tr-sup-1}
        \arrow[rr, Rightarrow, "\textcolor{red}{$\Tr$}"]
        \arrow[rd, Rightarrow, "\textcolor{blue}{$\eta_\otimes$}"']
        \&\&
        \xusebox{double-verytallid}
        \arrow[rr, Rightarrow, "$\eta_\otimes$"]
        \&\&
        \xusebox{eta-prime-target}
        \\
        \&
        \xusebox{tr-sup-2}
        \arrow[rr, Rightarrow, "\textcolor{red}{$\alpha$}", "\textcolor{red}{$\alpha$}"']
        \&\&
        \xusebox{tr-sup-3}
        \arrow[ru, Rightarrow, "\textcolor{red}{$\Tr$}"', end anchor={south west}]
        \&
      \end{tikzcd}
      \label{eq:tr-superposing}\tag{$\Tr$-sup}
      \end{equation}
    \end{minipage}
    \begin{minipage}{0.39\textwidth}
        \begin{equation}
        \begin{tikzcd}[math mode=false, sep=tiny, row sep=-25pt, ampersand replacement=\&]
          \xusebox{tr-van-i-1}
          \arrow[r, phantom, "$\Rightarrow$"]
          \arrow[rrr, Rightarrow, bend right, "\textcolor{blue}{$\rho$}", "\textcolor{blue}{$\rho$}"', start anchor={south}, end anchor={south west}]
          \&
          \xusebox{tr-van-i-2}
          \arrow[r, Rightarrow, "\textcolor{red}{$\psi_I$}"]
          \&
          \xusebox{tr-source}
          \arrow[r, Rightarrow, "$\Tr$"]
          \&
          \xusebox{veryverytallid}
        \end{tikzcd}
        \label{eq:tr-vanishing-identity}\tag{$\Tr$-van-$I$}
        \end{equation}
    \end{minipage}

      \begin{equation}
        \begin{tikzcd}[math mode=false, sep=tiny, row sep=-90pt, ampersand replacement=\&]
          \&
          \xusebox{tr-van-tensor-1}
          \arrow[rr, Rightarrow, "\textcolor{red}{$\Tr$}"]
          \arrow[ld, Rightarrow, "\textcolor{blue}{$\alpha$}", "\textcolor{blue}{$\alpha$}"', shorten=-2.5pt]
          \&
          \&
          \xusebox{tr-source}
          \arrow[rr, Rightarrow, "$\Tr$"]
          \&
          \&
          \xusebox{veryverytallid}
          \\
          \xusebox{tr-van-tensor-2}
          \arrow[rr, Rightarrow, bend right]
          \&
          \&
          \xusebox{tr-van-tensor-3}
          \arrow[rr, Rightarrow, "\textcolor{red}{$\varepsilon_\otimes$}"']
          \&
          \&
          \xusebox{tr-source}
          \arrow[ru, Rightarrow, "$\Tr$"', end anchor={[yshift=10pt]south west}]
          \&
        \end{tikzcd}
        \label{eq:tr-vanishing-tensor}\tag{$\Tr$-van-$\otimes$}
      \end{equation}
  \end{itemize}
\end{definition}

The following theorem records the fact that our presentation $\mathcal{T}$ correctly encodes traced monoidal categories, up to Cauchy completeness. A nice detail of the proof is certain traced monoidal category axioms arise \enquote{for free}, simply because the 2-morphism $\mathsf{Tr}$ is a natural transformation of profunctors.
\begin{theorem}\label[theorem]{traced-presentation-interpretation}
  Interpretations of $\mathcal{T}$ in $\Prof$ correspond to Cauchy complete
  traced monoidal categories.
\end{theorem}

An additional consideration is that a balanced traced monoidal category has
additional equations it must satisfy, and these are small deviations from
straightforwardly replacing $\mathcal{T}$ with its balanced version (the
presentation obtained by substituting $\mathcal{L}$ for $\mathcal{M}$ in
\Cref{traced-pseudomonoid}).

\begin{definition}
  \begin{xlrbox}{tr-yank-1}
    \begin{tikzpicture}[baseline=(current bounding box.center)]
      \node[identity] (i) {};
      \node[braid, anchor=bottomleftleg] (b) at (i.top) {};
      \node[identity, anchor=bottom] at (b.topleftleg) {};
      \node[2Dcup, anchor=leftleg] (cup) at (b.bottomrightleg) {};
      \node[2Dcap, anchor=leftleg] (cap) at (b.toprightleg) {};
      \node[identity, down, anchor=bottom] at (cup.rightleg) {};
      \node[draw=red, dashed, fit={(b.topleftleg) (b.toprightleg)}]{};
    \end{tikzpicture}
  \end{xlrbox}

  \begin{xlrbox}{tr-yank-2}
    \begin{tikzpicture}[baseline=(current bounding box.center), decoration={
        markings, mark=at position 0.5 with {\arrowreversed[black]{Stealth[length=1mm]}}
      }]
      \node[comult, dot=black] (c) {};
      \node[mult, anchor=top] (m) at (c.bottom) {};
      \node[braid, anchor=toprightleg] (b) at (m.rightleg) {};
      \node[2Dcup, anchor=leftleg] (cup) at (b.bottomrightleg) {};
      \node[2Dcap, anchor=leftleg] (cap) at (c.rightleg) {};
      \idstack[anchor=b.bottomleftleg, direction=down]{1}
      \idstack[anchor=c.leftleg, direction=up]{1}
      \draw[postaction={decorate}] (cup.rightleg) -- (cap.rightleg);
      \node[draw=red, dashed, fit={(b.bottomleftleg) (b.bottomrightleg) (m.top)}]{};
    \end{tikzpicture}
  \end{xlrbox}

  The \emph{balanced traced pseudomonoid presentation} $\mathcal{T}^\prime$ is
  given by the balanced version of $\mathcal{T}$, and additionally the
  equation:

  
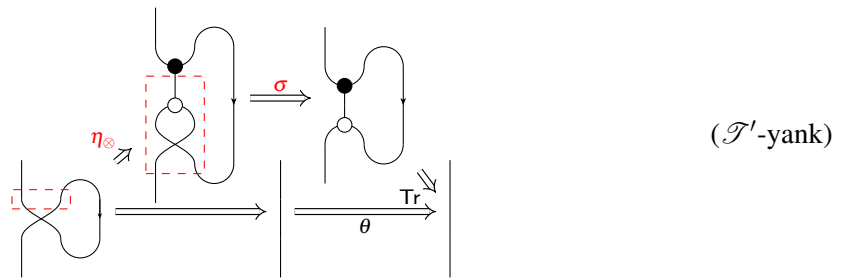
\begin{equation}
    \begin{tikzcd}[math mode=false, sep=tiny, row sep=-25pt]
      &
      \xusebox{tr-yank-2}
      \arrow[rr, Rightarrow, "\textcolor{red}{$\sigma$}"]
      &&
      \xusebox{tr-source}
      \arrow[dr, Rightarrow, "$\Tr$"']
      &
      \\
      \xusebox{tr-yank-1}
      \arrow[ru, Rightarrow, "\textcolor{red}{$\eta_\otimes$}"]
      \arrow[rr, Rightarrow]
      &&
      \xusebox{verytallid}
      \arrow[rr, Rightarrow, "$\theta$"']
      &&
      \xusebox{verytallid}
    \end{tikzcd}\label{eq:tr-yank}\tag{$\mathcal{T}^\prime$-yank}
  \end{equation}
\end{definition}


\begin{theorem}\label[theorem]{balanced-traced-presentation-interpretation}
  Interpretations of $\mathcal{T}^\prime$ in $\Prof$ correspond to Cauchy
  complete balanced traced monoidal categories.
\end{theorem}

The symmetric version is a special case of this where $\theta$ is given by the
identity $2$-morphism, as per \Cref{symmetry-balanced}.

\section{Braided autonomous categories are traced}\label{sec:autonomous-traced}

It is known that every tortile (autonomous + pivotal + balanced) category
admits a canonical trace \autocite[\S~3]{joyalTracedMonoidalCategories1996}.
This result is known to generalise to any braided autonomous category, which we
replicate in our framework. Furthermore, we justify that such a trace is not
necessarily unique.

The key observation comes from the fact that in the braided autonomous
pseudomonoid presentation, we have a chain of isomorphisms:
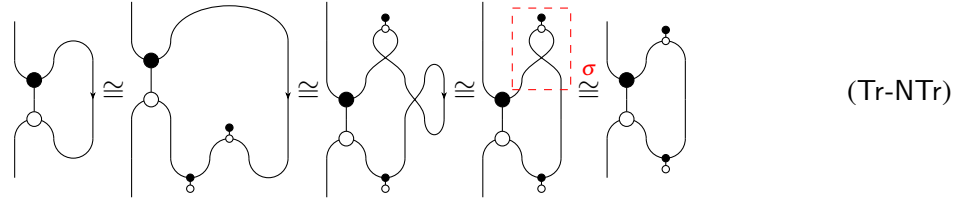
\begin{equation}
  \xusebox{tr-source}
  \cong
  \begin{tikzpicture}[baseline=(current bounding box.center), decoration={
          markings, mark=at position 0.5 with {\arrowreversed[black]{Stealth[length=1mm]}}
        }]
    \node[comult, dot=black] (c) {};
    \node[mult, anchor=top] (m) at (c.bottom) {};
    \idstack[anchor=m.rightleg]{1}
    \node[frobcup, autonomous, anchor=leftleg] (cup) at (id1.bottom) {};
    \node[frobcap, autonomous, anchor=leftleg] (cap) at (cup.rightleg) {};
    \node[2Dcup, anchor=leftleg] (cup') at (cap.rightleg) {};
    \node[2Dcap, anchor=leftleg, span=3] (cap') at (c.rightleg) {};
    \idstack[anchor=m.leftleg]{2}
    \idstack[anchor=c.leftleg, direction=up]{1}
    \draw[postaction={decorate}] (cup'.rightleg) -- (cap'.rightleg);
  \end{tikzpicture}
  \cong
  \begin{tikzpicture}[baseline=(current bounding box.center), decoration={
          markings, mark=at position 0.5 with {\arrowreversed[black]{Stealth[length=1mm]}}
        }]
    \node[comult, dot=black] (c) {};
    \node[mult, anchor=top] (m) at (c.bottom) {};
    \node[frobcup, autonomous, anchor=leftleg] (cup) at (m.rightleg) {};
    \node[twfrobcap, autonomous, anchor=leftleg] (cap) at (c.rightleg) {};
    \node[braid, anchor=topleftleg, span=0.5] (b) at (cap.rightleg) {};
    \node[2Dcup, anchor=leftleg, span=0.5] (cup') at (b.bottomrightleg) {};
    \node[2Dcap, anchor=leftleg, span=0.5] (cap') at (b.toprightleg) {};
    \idstack[anchor=m.leftleg]{1}
    \idstack[anchor=b.bottomleftleg]{1}
    \idstack[anchor=c.leftleg, direction=up]{2}
    \draw[postaction={decorate}] (cup'.rightleg) -- (cap'.rightleg);
  \end{tikzpicture}
  \cong
  \begin{tikzpicture}[baseline=(current bounding box.center)]
    \node[comult, dot=black] (c) {};
    \node[mult, anchor=top] (m) at (c.bottom) {};
    \node[twfrobcap, autonomous, anchor=leftleg] (cap) at (c.rightleg) {};
    \node[frobcup, autonomous, anchor=leftleg] (c') at (m.rightleg) {};
    \idstack[anchor=c'.rightleg, direction=up] {2}
    \idstack[anchor=m.leftleg, direction=down] {1}
    \idstack[anchor=c.leftleg, direction=up] {2}
    \node[draw=red, dashed, fit={(cap.rightleg) (cap.leftleg) (cap.top)}]{};
  \end{tikzpicture}
  \overset{\textcolor{red}{\sigma}}{\cong}
  \xusebox{ntr-source}
  \label{eq:trace-to-nearly}\tag{$\Tr$-$\NTr$}
\end{equation}
The first isomorphism utilises the duality generated by
$(\tinymorphism{frobcup, autonomous}, \tinymorphism{frobcap, autonomous})$ in
the autonomous pseudomonoid presentation, the second is an isotopy, the third
utilises the compact structure of $\Prof$, and the fourth utilises braiding.

Secondly, in any monoidal category, we have the following $2$-morphism.

\begin{definition}
For a monoidal category, its \emph{nearly tracing} is the following 2-morphism:
  \begin{equation}\label{eq:nearly-tracing}\tag{$\NTr$}
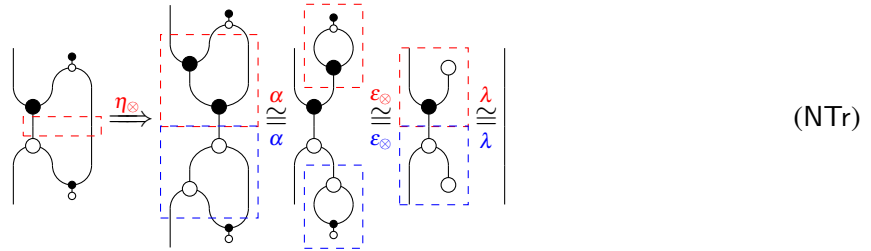
    \begin{tikzpicture}[baseline=(current bounding box.center)]
      \node[comult, dot=black] (c) {};
      \node[mult, anchor=top] (m) at (c.bottom) {};
      \node[frobcap, autonomous, anchor=leftleg] at (c.rightleg) {};
      \node[frobcup, autonomous, anchor=leftleg] (c') at (m.rightleg) {};
      \idstack[anchor=m.leftleg, direction=down] {1}
      \idstack[anchor=c.leftleg, direction=up] {1}
      \idstack[anchor=c'.rightleg, direction=up] {2}
      \node[draw=red, dashed, fit={(m.top) (id1.top)}]{};
    \end{tikzpicture}
    \xRightarrow{\textcolor{red}{\eta_\otimes}}
    \begin{tikzpicture}[baseline=(current bounding box.center)]
      \node[comult, dot=black] (c) {};
      \node[comult, dot=black, anchor=leftleg, span=1.5] (c') at (c.bottom) {};
      \node[mult, anchor=top, span=1.5] (m') at (c'.bottom) {};
      \node[mult, anchor=top] (m) at (m'.leftleg) {};
      \node[frobcap, autonomous, anchor=leftleg] (cap) at (c.rightleg) {};
      \node[frobcup, autonomous, anchor=leftleg] (cup) at (m.rightleg) {};
      \idstack[anchor=m.leftleg, direction=down] {1}
      \idstack[anchor=c.leftleg, direction=up] {1}
      \idstack[anchor=cup.rightleg, direction=up] {1}
      \idstack[anchor=cap.rightleg, direction=down] {1}
      \node[draw=red, dashed, fit={(c'.south) (c'.rightleg) (c.leftleg)}]{};
      \node[draw=blue, dashed, fit={(m'.north) (m'.rightleg) (m.leftleg)}]{};
    \end{tikzpicture}
    \underset{\textcolor{blue}{\alpha}}{\overset{\textcolor{red}{\alpha}}{\cong}}
    \begin{tikzpicture}[baseline=(current bounding box.center)]
      \node[comult, dot=black] (c) {};
      \node[mult, anchor=top] (m) at (c.bottom) {};
      \node[comult, dot=black, anchor=bottom] (c') at (c.rightleg) {};
      \node[mult, anchor=top] (m') at (m.rightleg) {};
      \node[frobcap, autonomous, anchor=leftleg] (cap) at (c'.leftleg) {};
      \node[frobcup, autonomous, anchor=leftleg] (cup) at (m'.leftleg) {};
      \idstack[anchor=m.leftleg, direction=down] {1}
      \idstack[anchor=c.leftleg, direction=up] {1}
      \node[draw=red, dashed, fit={(c'.south) (cap.top) (cap.rightleg) (cap.leftleg)}]{};
      \node[draw=blue, dashed, fit={(m'.north) (cup.bottom) (cup.rightleg) (cup.leftleg)}]{};
    \end{tikzpicture}
    \underset{\textcolor{blue}{\varepsilon_\otimes}}{\overset{\textcolor{red}{\varepsilon_\otimes}}{\cong}}
    \begin{tikzpicture}[baseline=(current bounding box.center)]
      \node[comult, dot=black] (c) {};
      \node[counit, anchor=bottom] (cu) at (c.rightleg) {};
      \node[mult, anchor=top] (m) at (c.bottom) {};
      \node[unit, anchor=top] (u) at (m.rightleg) {};
      \idstack[anchor=m.leftleg, direction=down] {1}
      \idstack[anchor=c.leftleg, direction=up] {1}
      \node[draw=red, dashed, fit={(c.south) (c.leftleg) (cu)}]{};
      \node[draw=blue, dashed, fit={(m.north) (m.leftleg) (u)}]{};
    \end{tikzpicture}
    \underset{\textcolor{blue}{\lambda}}{\overset{\textcolor{red}{\lambda}}{\cong}}
    \xusebox{veryverytallid}
  \end{equation}
\end{definition}

It is then reasonable to come up with analogues to the tracing axioms, and ask
if \Cref{eq:nearly-tracing} satisfies them.

\begin{proposition}
  $\NTr$ satisfies the nearly-tracing axioms
  for any Cauchy complete monoidal category.
\end{proposition}

\begin{corollary}
Any Cauchy complete category for which the following isomorphism exists, can be equipped with a trace:
  \begin{equation*}
    \xusebox{tr-source}
    \cong
    \xusebox{ntr-source}
  \end{equation*}
\end{corollary}

\noindent
Thanks to \Cref{eq:trace-to-nearly}, this implies the standard result that a braided monoidal category is traced.

\begin{remark}
  This trace is not necessarily unique, because we can arbitrarily apply a
  twist $2$-morphism $\theta$ for any available twisting. By fixing a chosen
  twist (in the sense of a balanced monoidal category), the trace becomes
  canonical as in \textcite[\S~3]{joyalTracedMonoidalCategories1996}.
\end{remark}

\section{On traced \texorpdfstring{$*$}{*}-autonomous categories}\label{sec:*-autonomous-categories}


In this section we apply our traced pseudomonoid technology to study phenomena
related to $*$-autonomous categories. Firstly, we show that for a Cauchy
complete $*$-autonomous category, a right $\otimes$-trace is equivalent to a
left $\invamp$-trace. Secondly, inspired by the result of
\textcite{hajgatoTracedAutonomousCategories2013} that every traced symmetric
$*$-autonomous category is compact closed, we seek to investigate the
non-symmetric version of this statement. We use our traced pseudomonoid
approach to derive a sufficient condition for a traced
$*$-autonomous category to be autonomous.

\subsection{\texorpdfstring{$*$}{*}-Autonomous categories}

\begin{definition}
  The \emph{Frobenius pseudomonoid presentation} $\mathcal{F}$ is obtained by combining the
  pseudomonoid presentation $(\cdot, \xusebox{tinymonoid-black},
  \xusebox{tinyunit-black})$ with the pseudocomonoid presentation $(\cdot,
  \xusebox{tinycomonoid-black}, \xusebox{tinycounit-black})$ on the same object, and additionally:
  \begin{itemize}
    \item invertible generating $2$-morphisms:

    \begin{equation}\label{eq:frobenius}\tag{Frob}
      \begin{tikzpicture}[baseline=(current bounding box.center)]
        \node[mult, dot=black] (m) {};
        \node[comult, anchor=leftleg, dot=black] (c) at (m.rightleg) {};
        \node[identity, anchor=top] at (m.leftleg) {};
        \node[identity, anchor=bottom] at (c.rightleg) {};
      \end{tikzpicture}
      \cong
      \begin{tikzpicture}[baseline=(current bounding box.center)]
        \node[mult, dot=black] (m) {};
        \node[comult, dot=black, anchor=bottom] at (m.top) {};
      \end{tikzpicture}
      \cong
      \begin{tikzpicture}[baseline=(current bounding box.center)]
        \node[comult, dot=black] (c) {};
        \node[mult, anchor=leftleg, dot=black] (m) at (c.rightleg) {};
        \node[identity, anchor=top] at (m.rightleg) {};
        \node[identity, anchor=bottom] at (c.leftleg) {};
      \end{tikzpicture}
    \end{equation}
    \item equational structure making these $2$-morphisms coherent.
  \end{itemize}
\end{definition}

\noindent\textcite[Definition~1.2]{dunnCoherenceFrobeniusPseudomonoids2016} give an
explicit presentation which they prove to be coherent, with equational structure given by the so-called \enquote{swallowtail equations}, which we omit here for brevity, using coherence directly to derive our results. In this case, the coherence result states that, given two parallel 2-morphisms $P,Q$, whose common source 1-morphism is connected and acyclic as a string diagram, then $P=Q$.

As above, we are interested in the presentation obtained by freely adding right
adjoints to particular generating $1$-morphisms.

\begin{definition}
  The \emph{right-adjoint Frobenius pseudomonoid presentation} $\mathcal{F}^*$
  is obtained from $\mathcal{F}$ by adding right-adjoint generating $1$-morphisms for the
  pseudomonoid multiplication and unit: $\xusebox{tinymonoid-black} \dashv
  \xusebox{tinycomonoid-white}$ and $\xusebox{tinyunit-black} \dashv
  \xusebox{tinycounit-white}$.
\end{definition}

  This structure corresponds to Cauchy complete $*$-autonomous categories,
  which we outline next. Full details can be found in
  \textcite[\S~2.7]{dunnCoherenceFrobeniusPseudomonoids2016}.

\begin{theorem}\label[theorem]{frobenius-presentation-interpretation}
  Interpretations of $\mathcal{F}^*$ in $\Prof$ correspond to Cauchy
  complete $*$-autonomous categories, where the generating $1$-morphisms
  $\xusebox{tinymonoid-black}$ and $\xusebox{tinyunit-black}$ represent
  $\invamp$ and $\bot$ respectively, and the derived $1$-morphisms
  $\xusebox{tinymonoid-white}$ and $\xusebox{tinyunit-white}$
  \autocite[Definition~2.33]{dunnCoherenceFrobeniusPseudomonoids2016}
  represent $\otimes$ and $I$ respectively.
\end{theorem}

\begin{definition}
  The \emph{traced $*$-autonomous presentation} $\mathcal{T}^*$ is obtained by
  combining $\mathcal{T}$ and $\mathcal{F}^*$, i.e.\ the presentation
  containing a right-adjoint Frobenius pseudomonoid as in $\mathcal{F}^*$,
  where the derived left-adjoint pseudomonoid representing the tensor product
  $(\cdot, \xusebox{tinymonoid-white}, \xusebox{tinyunit-white})$ is
  additionally a traced pseudomonoid.
\end{definition}

\noindent
Its constituent parts are interpreted in $\Prof$ by Cauchy complete traced and
$*$-autonomous categories respectively, so the combined presentation is
interpreted by a category which is simultaneously traced and $*$-autonomous.

\subsection{Rotations}

In this section we show that for a $*$-autonomous category, a left
$\otimes$-trace is equivalent to a right $\invamp$-trace. The idea is that
$\otimes$ and $\invamp$ are related by duality, and that tracing can be
transported through this duality. Furthermore, the dual trace obtained is
\enquote{rotated}. Before proving this, we will first try to simplify
$\mathcal{T}^*$.

A complication is that $\mathcal{T}^*$ refers to
composites containing the $1$-morphism $\xusebox{tinymonoid-white}$, which with
respect to $\mathcal{F}^*$ is a derived $1$-morphism, as opposed to a
generating $1$-morphism. To simplify this, we shall progressively
rewrite the data of $\mathcal{T}^*$ in terms of generating $1$-morphisms. First, we dispense with the compact closed assumption in the ambient
bicategory ($\Prof$).

\begin{remark}\label[remark]{rotate-right-trace}
  We utilise the Frobenius duality generated by $(\tinymorphism{frobcup},
  \tinymorphism{frobcap})$, similarly to \Cref{eq:trace-to-nearly}, obtaining the
  isomorphism:
  \begin{equation*}
    \xusebox{tr-source}
    \cong
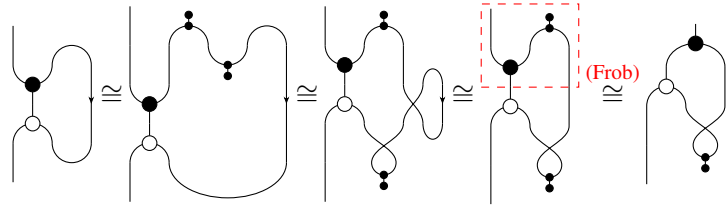
    \begin{tikzpicture}[baseline=(current bounding box.center), decoration={
            markings, mark=at position 0.5 with {\arrowreversed[black]{Stealth[length=1mm]}}
          }]
      \node[comult, dot=black] (c) {};
      \node[mult, anchor=top] (m) at (c.bottom) {};
      \idstack[anchor=c.rightleg, direction=up]{1}
      \node[frobcap, anchor=leftleg] (cap) at (id1.top) {};
      \node[frobcup, anchor=leftleg] (cup) at (cap.rightleg) {};
      \node[2Dcup, anchor=leftleg, span=3] (cup') at (m.rightleg) {};
      \node[2Dcap, anchor=leftleg] (cap') at (cup.rightleg) {};
      \idstack[anchor=m.leftleg, direction=down]{1}
      \idstack[anchor=c.leftleg, direction=up]{2}
      \draw[postaction={decorate}] (cup'.rightleg) -- (cap'.rightleg);
    \end{tikzpicture}
    \cong
    \begin{tikzpicture}[baseline=(current bounding box.center), decoration={
            markings, mark=at position 0.5 with {\arrowreversed[black]{Stealth[length=1mm]}}
          }]
      \node[comult, dot=black] (c) {};
      \node[mult, anchor=top] (m) at (c.bottom) {};
      \node[twfrobcup, anchor=leftleg] (cup) at (m.rightleg) {};
      \node[frobcap, anchor=leftleg] (cap) at (c.rightleg) {};
      \node[braid, anchor=bottomleftleg, span=0.5] (b) at (cup.rightleg) {};
      \node[2Dcup, anchor=leftleg, span=0.5] (cup') at (b.bottomrightleg) {};
      \node[2Dcap, anchor=leftleg, span=0.5] (cap') at (b.toprightleg) {};
      \idstack[anchor=m.leftleg]{2}
      \idstack[anchor=b.topleftleg, direction=up]{1}
      \idstack[anchor=c.leftleg, direction=up]{1}
      \draw[postaction={decorate}] (cup'.rightleg) -- (cap'.rightleg);
    \end{tikzpicture}
    \cong
    \begin{tikzpicture}[baseline=(current bounding box.center)]
      \node[comult, dot=black] (c) {};
      \node[mult, anchor=top] (m) at (c.bottom) {};
      \node[frobcap, anchor=leftleg] (m') at (c.rightleg) {};
      \node[twfrobcup, anchor=leftleg] at (m.rightleg) {};
      \idstack[anchor=m'.rightleg, direction=down] {2}
      \idstack[anchor=m.leftleg, direction=down] {2}
      \idstack[anchor=c.leftleg, direction=up] {1}
      \node[draw=red, dashed, fit={(c.south) (c.leftleg) (m'.top) (m'.rightleg)}]{};
    \end{tikzpicture}
    \overset{\textcolor{red}{\labelcref{eq:frobenius}}}{\cong}
    \xusebox{rtr-source}
  \end{equation*}
  Notice that such a $1$-morphism may be interpreted in any monoidal bicategory
  with duals, as opposed to the stronger requirement of compactness.
  Henceforth, we will derive a tracing presentation with respect to this
  $1$-morphism.
\end{remark}

\begin{definition}
  \begin{xlrbox}{rtr-van-i-1}
    \begin{tikzpicture}[baseline=(current bounding box.center)]
      \node[mult, dot=black] (c) {};
      \node[mult, anchor=top] (m) at (c.leftleg) {};
      \node[unit, dot=black, anchor=top] (u) at (c.rightleg) {};
      \node[unit, anchor=top] (u') at (m.rightleg) {};
      \idstack[anchor=m.leftleg, direction=down]{1}
      \node[draw=blue, dashed, fit={(u') (m.north) (m.leftleg)}]{};
      \node[draw=blue, dashed, fit={(u) (c.south) (c.leftleg)}]{};
    \end{tikzpicture}
  \end{xlrbox}

  \begin{xlrbox}{rtr-van-i-2}
    \begin{tikzpicture}[baseline=(current bounding box.center)]
      \node[mult, dot=black, span=1.5] (c) {};
      \node[mult, anchor=top] (m) at (c.leftleg) {};
      \idstack[anchor=c.rightleg] {3}
      \node[twfrobcup, dot=black, anchor=rightleg] (cup) at (id3.bottom) {};
      \node[unit, anchor=top] (u) at (m.rightleg) {};
      \node[counit, dot=black, anchor=bottom] (cu) at (cup.leftleg) {};
      \idstack[anchor=m.leftleg, direction=down]{4}
      \node[draw=red, dashed, fit={(cu) (u)}]{};
    \end{tikzpicture}
  \end{xlrbox}

  \begin{xlrbox}{rtr-van-tensor-1}
    \begin{tikzpicture}[baseline=(current bounding box.center)]
      \node[mult] (m) {};
      \node[mult, dot=black, anchor=leftleg, span=1.5] (m') at (m.top) {};
      \node[twfrobcup, anchor=leftleg] (cup) at (m.rightleg) {};
      \node[identity, anchor=top] at (m'.rightleg) {};
      \idstack[anchor=m.leftleg]{1}
      \node[mult, anchor=top] (m'') at (id1.bottom) {};
      \node[mult, dot=black, anchor=leftleg] (m''') at (m'.top) {};
      \idstack[anchor=m'''.rightleg]{4}
      \node[twfrobcup, anchor=leftleg, span=1.75] (cup') at (m''.rightleg) {};
      \draw (cup'.rightleg) -- (id4.bottom);
      \idstack[anchor=m''.leftleg]{2}
      \node[draw=red, dashed, fit={(m.leftleg) (m'.top) (cup.bottom) (cup.rightleg)}]{};
      \node[draw=blue, dashed, fit={(m.north) (m.rightleg) (m''.leftleg)}]{};
      \node[draw=blue, dashed, fit={(m'''.north) (m'.leftleg) (m'''.rightleg)}]{};
    \end{tikzpicture}
  \end{xlrbox}

  \begin{xlrbox}{rtr-van-tensor-2}
    \begin{tikzpicture}[baseline=(current bounding box.center)]
      \node[mult] (m) {};
      \node[mult, dot=black, anchor=leftleg, span=2.5] (m') at (m.top) {};
      \node[mult, anchor=top] (m'') at (m.rightleg) {};
      \idstack[anchor=m'.rightleg]{1}
      \node[mult, dot=black, anchor=top] (m''') at (id1.bottom) {};
      \node[twfrobcup, anchor=leftleg] at (m''.rightleg) {};
      \node[identity, anchor=top] (id) at (m''.leftleg) {};
      \node[swishr, anchor=top] (swish) at (id.bottom) {};
      \node[twfrobcup, anchor=leftleg, span=2] (cup') at (swish.bottom) {};
      \idstack[anchor=m.leftleg]{5}
      \idstack[direction=up, anchor=cup'.rightleg]{2}
    \end{tikzpicture}
  \end{xlrbox}

  \begin{xlrbox}{rtr-van-tensor-3}
    \begin{tikzpicture}[baseline=(current bounding box.center)]
      \node[mult] (m) {};
      \node[mult, dot=black, anchor=leftleg, span=1.5] (m') at (m.top) {};
      \node[mult, anchor=top] (m'') at (m.rightleg) {};
      \node[comult, dot=black, anchor=leftleg] (c) at (m''.leftleg) {};
      \idstack[anchor=m'.rightleg]{3}
      \node[twfrobcup, anchor=leftleg] at (c.bottom) {};
      \idstack[anchor=m.leftleg]{4}
      \node[draw=red, dashed, fit={(m''.rightleg) (m''.leftleg) (c.south) (m''.north)}]{};
    \end{tikzpicture}
  \end{xlrbox}

  \begin{xlrbox}{rtr-sup-1}
    \begin{tikzpicture}[baseline=(current bounding box.center)]
      \node[mult] (m) {};
      \node[mult, dot=black, anchor=leftleg, span=1.5] (m') at (m.top) {};
      \node[twfrobcup, anchor=leftleg] (cup) at (m.rightleg) {};
      \node[identity, anchor=top] at (m'.rightleg) {};
      \idstack[anchor=m.leftleg]{2}
      \node[identity, xshift=-2\cobwidth-2\cobgap, anchor=top] (i) at (m'.top) {};
      \node[identity, xshift=-0.75\cobwidth-0.75\cobgap, anchor=bottom] (i') at (id2.bottom) {};
      \draw (i.bottom) -- (i'.top);
      \node[draw=red, dashed, fit={(m'.rightleg) (cup.bottom) (m.leftleg) (m'.north)}]{};
      \node[draw=blue, dashed, fit={(m.top) (i.bottom)}]{};
    \end{tikzpicture}
  \end{xlrbox}

  \begin{xlrbox}{rtr-sup-2}
    \begin{tikzpicture}[baseline=(current bounding box.center)]
      \node[mult] (m) {};
      \node[mult, anchor=rightleg] (m'') at (m.top) {};
      \node[comult, dot=black, anchor=bottom] (c) at (m''.top) {};
      \node[mult, dot=black, anchor=leftleg, span=1.5] (m') at (c.rightleg) {};
      \node[twfrobcup, anchor=leftleg] (cup) at (m.rightleg) {};
      \idstack[anchor=m'.rightleg]{3}
      \idstack[anchor=m.leftleg]{2}
      \idstack[anchor=m''.leftleg]{3}
      \node[identity, xshift=-1.75\cobwidth-1.75\cobgap, anchor=top] (i) at (m'.top) {};
      \draw (i.bottom) -- (c.leftleg);
      \node[draw=red, dashed, fit={(c.south) (c.leftleg) (m'.rightleg) (m'.north)}]{};
      \node[draw=blue, dashed, fit={(m''.north) (m''.leftleg) (m.rightleg)}]{};
    \end{tikzpicture}
  \end{xlrbox}

  \begin{xlrbox}{rtr-sup-3}
    \begin{tikzpicture}[baseline=(current bounding box.center)]
      \node[mult] (m) {};
      \node[mult, dot=black, anchor=leftleg, span=1.5] (m') at (m.top) {};
      \node[twfrobcup, anchor=leftleg] (cup) at (m.rightleg) {};
      \node[identity, anchor=top] at (m'.rightleg) {};
      \idstack[anchor=m.leftleg]{2}
      \node[mult, anchor=top] (m'') at (id2.bottom) {};
      \node[comult, dot=black, anchor=bottom] (c) at (m'.top) {};
      \node[draw=red, dashed, fit={(m'.north) (m'.rightleg) (m.leftleg) (cup.bottom)}]{};
    \end{tikzpicture}
  \end{xlrbox}

  The \emph{rotational right $\otimes$-traced $*$-autonomous pseudomonoid
  presentation} is given by the data of $\mathcal{F}^*$, and additionally:
  \begin{itemize}
    \item a generating $2$-morphism:

      \begin{equation}\label{eq:rtr}\tag{$\RTr$}
        \xusebox{rtr-source}
        \xRightarrow{\RTr}
        \xusebox{veryverytallid}
      \end{equation}
    \item equations witnessing the axioms of traced monoidal categories:

    \begin{minipage}{0.5\textwidth}
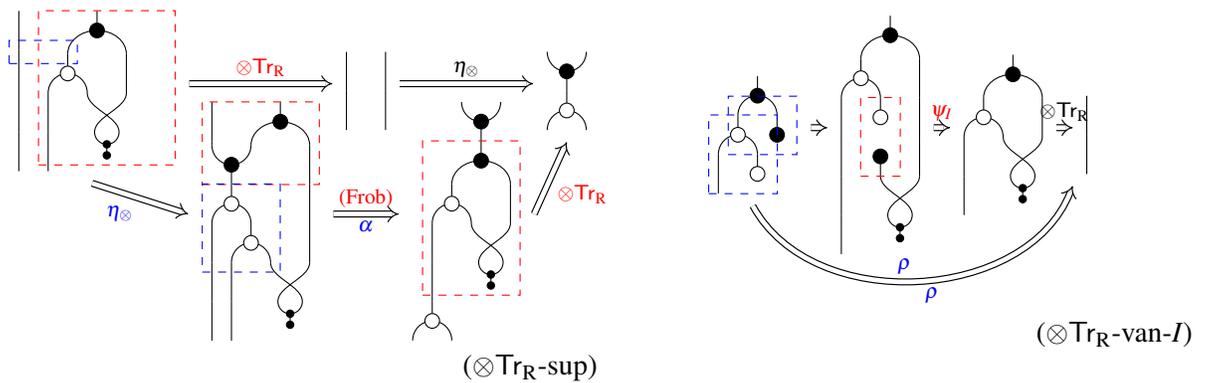
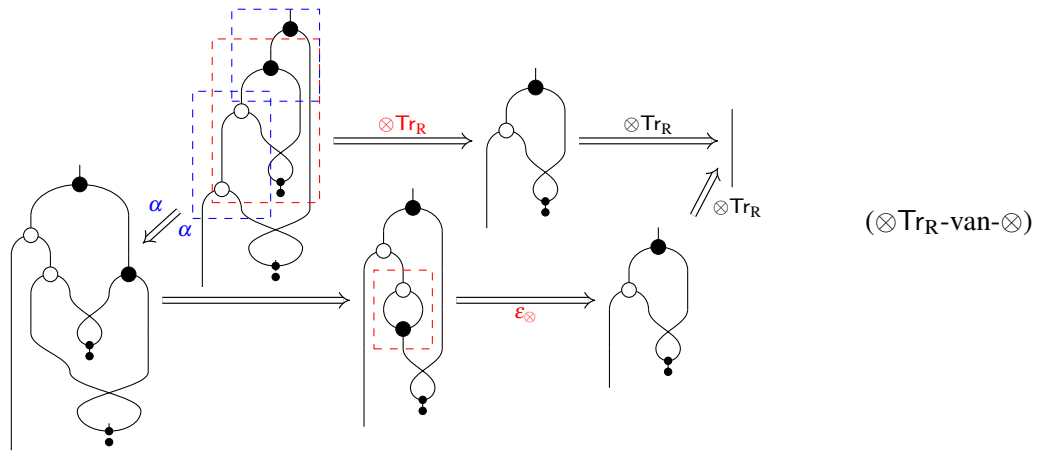
        \begin{equation}
        \begin{tikzcd}[math mode=false, sep=0pt, row sep=-35pt, ampersand replacement=\&]
          \xusebox{rtr-sup-1}
          \arrow[rd, Rightarrow, "\textcolor{blue}{$\eta_\otimes$}"', start anchor={south}, end anchor={west}]
          \arrow[rr, Rightarrow, "\textcolor{red}{$\RTr$}"]
          \&\&
          \xusebox{double-tallid}
          \arrow[rr, Rightarrow, "$\eta_\otimes$"]
          \&\&
          \xusebox{eta-prime-target}
          \\
          \&
          \xusebox{rtr-sup-2}
          \arrow[rr, Rightarrow, "\textcolor{red}{$\labelcref{eq:frobenius}$}", "\textcolor{blue}{$\alpha$}"']
          \&\&
          \xusebox{rtr-sup-3}
          \arrow[ru, Rightarrow, "\textcolor{red}{$\RTr$}"', start anchor={east}, end anchor={south}]
          \&
        \end{tikzcd}
        \label{eq:rtr-superposing}\tag{$\RTr$-sup}
        \end{equation}
        \end{minipage}
        \begin{minipage}{0.49\textwidth}
      \begin{equation}
        \begin{tikzcd}[math mode=false, sep=tiny, row sep=-25pt, ampersand replacement=\&]
          \xusebox{rtr-van-i-1}
          \arrow[r, Rightarrow]
          \arrow[rrr, bend right=65, Rightarrow, "\textcolor{blue}{$\rho$}", "\textcolor{blue}{$\rho$}"', start anchor={south}, end anchor={south west}]
          \&
          \xusebox{rtr-van-i-2}
          \arrow[r, Rightarrow, "\textcolor{red}{$\psi_I$}"]
          \&
          \xusebox{rtr-source}
          \arrow[r, Rightarrow, "$\RTr$"]
          \&
          \xusebox{tallid}
        \end{tikzcd}
        \label{eq:rtr-vanishing-identity}\tag{$\RTr$-van-$I$}
      \end{equation}
      \end{minipage}
      \begin{equation}
        \begin{tikzcd}[math mode=false, sep=tiny, row sep=-55pt, ampersand replacement=\&]
          \&
          \xusebox{rtr-van-tensor-1}
          \arrow[rr, Rightarrow, "\textcolor{red}{$\RTr$}"]
          \arrow[ld, Rightarrow, "\textcolor{blue}{$\alpha$}", "\textcolor{blue}{$\alpha$}"', shorten >=-10pt]
          \&
          \&
          \xusebox{rtr-source}
          \arrow[rr, Rightarrow, "$\RTr$"]
          \&
          \&
          \xusebox{tallid}
          \\
          \xusebox{rtr-van-tensor-2}
          \arrow[rr, Rightarrow]
          \&
          \&
          \xusebox{rtr-van-tensor-3}
          \arrow[rr, Rightarrow, "\textcolor{red}{$\varepsilon_\otimes$}"']
          \&
          \&
          \xusebox{rtr-source}
          \arrow[ru, Rightarrow, "$\RTr$"']
          \&
        \end{tikzcd}
        \label{eq:rtr-vanishing-tensor}\tag{$\RTr$-van-$\otimes$}
      \end{equation}
  \end{itemize}
\end{definition}

\begin{proposition}\label[theorem]{rotational-right-traced-presentation-interpretation}
  Interpretations of the rotational right $\otimes$-traced $*$-autonomous
  pseudomonoid presentation are Cauchy complete traced $*$-autonomous
  categories, where $\otimes$ is traced on the right.
\end{proposition}

\begin{proof}
  From \Cref{traced-presentation-interpretation}, it suffices to show that
  we can recover all the data of $\mathcal{T}$ from this presentation.
  This holds by transporting along the isomorphism defined in
  \Cref{rotate-right-trace}.
\end{proof}

We have weakened our setting to a symmetric monoidal bicategory with duals,
rather than a compact closed bicategory (notice the lack of cups and caps in
our string diagrams). However, they still mention $\xusebox{tinymonoid-white}$,
as we would like to discuss a traced monoidal category where the trace is with
respect to $\otimes$ (as opposed to the other tensor product $\invamp$).

\begin{definition}
  The \emph{rotational left $\invamp$-traced $*$-autonomous pseudomonoid
  presentation} is given by the data of $\mathcal{F}^*$, and additionally:
  \begin{itemize}
    \item a generating $2$-morphism:

      \begin{equation}\label{eq:ltr}\tag{$\LTr$}
        \xusebox{ltr-source}
        \xRightarrow{\LTr}
        \xusebox{veryverytallid}
      \end{equation}
    \item equations witnessing the axioms of traced monoidal categories,
      analogous to \ref{eq:rtr}.
  \end{itemize}
\end{definition}

\begin{proposition}
  Interpretations of the rotational left $\invamp$-traced $*$-autonomous
  pseudomonoid presentation are Cauchy complete traced $*$-autonomous
  categories, where
  $\invamp$ is traced on the left.
\end{proposition}

\begin{proof}
  Symmetric to the proof of
  \Cref{rotational-right-traced-presentation-interpretation}.
\end{proof}

\noindent
We can now state the main result of this section.

\begin{theorem}\label[theorem]{ltr-rtr}
  For a Cauchy complete $*$-autonomous category, a right $\otimes$-trace and a left $\invamp$-trace are equivalent.
\end{theorem}

\subsection{Invertible linear distributivity}

\begin{definition}
  A $*$-autonomous category has distinguished maps called \emph{linear
  distributors}; for all objects $A$, $B$, and $C$:
  \begin{equation*}
    A \otimes (B \invamp C) \xrightarrow{\delta_L} (A \otimes B) \invamp C, \quad
    (A \invamp B) \otimes C \xrightarrow{\delta_R} A \invamp (B \otimes C).
  \end{equation*}
  With respect to $\mathcal{F}^*$, these are corepresented by the composite
  $2$-morphisms:
  \begin{equation*}
    \delta_L \coloneqq
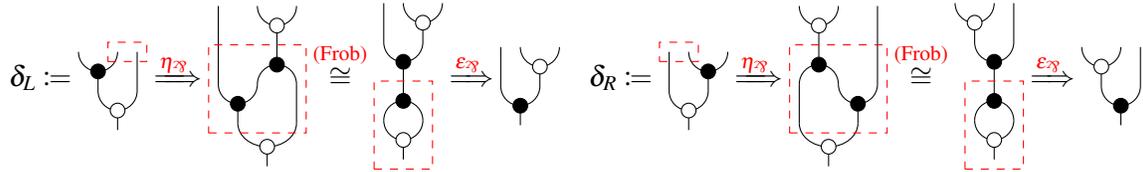
    \begin{tikzpicture}[baseline=(current bounding box.center)]
      \node[comult] (c) {};
      \node[comult, dot=black, anchor=bottom] (m) at (c.leftleg) {};
      \node[identity, anchor=bottom] (i) at (c.rightleg) {};
      \node[draw=red, dashed, fit={(m.rightleg) (i.top)}]{};
    \end{tikzpicture}
    \xRightarrow{\textcolor{red}{\eta_\invamp}}
    \begin{tikzpicture}[baseline=(current bounding box.center)]
      \node[comult, span=1.5] (c) {};
      \node[comult, dot=black, anchor=bottom] (m) at (c.leftleg) {};
      \node[identity, anchor=bottom] at (c.rightleg) {};
      \node[mult, dot=black, anchor=leftleg] (m') at (m.rightleg) {};
      \node[comult, anchor=bottom] (c') at (m'.top) {};
      \idstack[anchor=m.leftleg, direction=up]{2}
      \node[draw=red, dashed, fit={(m'.north) (m.leftleg) (m'.rightleg) (m.bottom)}]{};
    \end{tikzpicture}
    \overset{\textcolor{red}{\labelcref{eq:frobenius}}}{\cong}
    \begin{tikzpicture}[baseline=(current bounding box.center)]
      \node[comult, dot=black] (c) {};
      \node[comult, anchor=bottom] (m) at (c.rightleg) {};
      \node[mult, dot=black, anchor=top] (m') at (c.bottom) {};
      \node[comult, anchor=leftleg] (c') at (m'.leftleg) {};
      \idstack[anchor=c.leftleg, direction=up]{1}
      \node[draw=red, dashed, fit={(m'.north) (m'.leftleg) (m'.rightleg) (c'.bottom)}]{};
    \end{tikzpicture}
    \xRightarrow{\textcolor{red}{\varepsilon_\invamp}}
    \begin{tikzpicture}[baseline=(current bounding box.center)]
      \node[comult, dot=black] (c) {};
      \node[comult, anchor=bottom] (m) at (c.rightleg) {};
      \idstack[anchor=c.leftleg, direction=up]{1}
    \end{tikzpicture}
    \quad
    \delta_R \coloneqq
    \begin{tikzpicture}[baseline=(current bounding box.center)]
      \node[comult] (c) {};
      \node[comult, dot=black, anchor=bottom] (m) at (c.rightleg) {};
      \node[identity, anchor=bottom] (i) at (c.leftleg) {};
      \node[draw=red, dashed, fit={(m.leftleg) (i.top)}]{};
    \end{tikzpicture}
    \xRightarrow{\textcolor{red}{\eta_\invamp}}
    \begin{tikzpicture}[baseline=(current bounding box.center)]
      \node[comult, span=1.5] (c) {};
      \node[comult, dot=black, anchor=bottom] (m) at (c.rightleg) {};
      \node[identity, anchor=bottom] at (c.leftleg) {};
      \node[mult, dot=black, anchor=rightleg] (m') at (m.leftleg) {};
      \node[comult, anchor=bottom] (c') at (m'.top) {};
      \idstack[anchor=m.rightleg, direction=up]{2}
      \node[draw=red, dashed, fit={(m'.north) (m.rightleg) (m'.leftleg) (m.bottom)}]{};
    \end{tikzpicture}
    \overset{\textcolor{red}{\labelcref{eq:frobenius}}}{\cong}
    \begin{tikzpicture}[baseline=(current bounding box.center)]
      \node[comult, dot=black] (c) {};
      \node[comult, anchor=bottom] (m) at (c.leftleg) {};
      \node[mult, dot=black, anchor=top] (m') at (c.bottom) {};
      \node[comult, anchor=rightleg] (c') at (m'.rightleg) {};
      \idstack[anchor=c.rightleg, direction=up]{1}
      \node[draw=red, dashed, fit={(m'.north) (m'.rightleg) (m'.leftleg) (c'.bottom)}]{};
    \end{tikzpicture}
    \xRightarrow{\textcolor{red}{\varepsilon_\invamp}}
    \begin{tikzpicture}[baseline=(current bounding box.center)]
      \node[comult, dot=black] (c) {};
      \node[comult, anchor=bottom] (m) at (c.leftleg) {};
      \idstack[anchor=c.rightleg, direction=up]{1}
    \end{tikzpicture}
  \end{equation*}
\end{definition}

\begin{definition}
  The \emph{invertibly linear distributive presentation} $\mathcal{D}$ is
  obtained by adding inverses to the linear distributor $2$-morphisms in
  $\mathcal{F}^*$.
\end{definition}

\noindent
This is simpler to work with, and is equivalent to the data of
\Cref{autonomous} by bending the open leg of $\xusebox{tinymonoid-black}$
with $\tinymorphism{frobcup}$.

Recall that an autonomous category is precisely a $*$-autonomous category
which has invertible linear distributors.
Here we derive a white \enquote{Frobenius} $2$-morphism, from $\RTr$ and
$\RTrPar$ (equivalently, $\LTrTensor$), and find two equations we would like it
to satisfy. For brevity, our aim is to show that $\delta_R$ inverts, but for
$\delta_L$ the symmetric \enquote{Frobenius} $2$-morphism and associated
conditions are required.

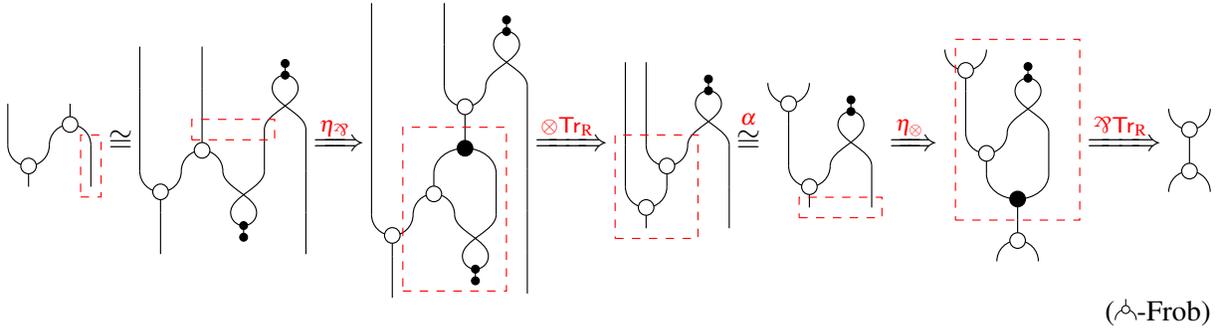
\begin{definition}
  \begin{equation}\label{eq:white-frobenius}\tag{$\xusebox{tinymonoid-white}$-Frob}
    \resizebox{\textwidth}{!}{$
        \begin{tikzpicture}[baseline=(current bounding box.center)]
          \node[comult] (c) {};
          \node[mult, anchor=leftleg] (m) at (c.rightleg) {};
          \idstack[anchor=c.leftleg, direction=up]{1}
          \idstack[anchor=m.rightleg, direction=down]{1}
          \node[draw=red, dashed, fit={(id1.top) (id1.bottom)}]{};
        \end{tikzpicture}
        \cong
        \begin{tikzpicture}[baseline=(current bounding box.center)]
          \node[comult] (c) {};
          \idstack[anchor=c.bottom, direction=down]{1}
          \node[mult, anchor=leftleg] (m) at (c.rightleg) {};
          \idstack[anchor=m.top, direction=up]{2}
          \idstack[anchor=c.leftleg, direction=up]{3}
          \node[twfrobcup, anchor=leftleg] (cup) at (m.rightleg) {};
          \idstack[anchor=cup.rightleg, direction=up]{1}
          \node[twfrobcap, anchor=leftleg] (cap) at (id1.top) {};
          \node[draw=red, dashed, fit={(m.top) (id1.top)}]{};
          \idstack[anchor=cap.rightleg, direction=down]{3}
        \end{tikzpicture}
        \xRightarrow{\textcolor{red}{\eta_\invamp}}
        \begin{tikzpicture}[baseline=(current bounding box.center)]
          \node[comult] (c) {};
          \idstack[anchor=c.bottom, direction=down]{1}
          \node[mult, anchor=leftleg] (m) at (c.rightleg) {};
          \node[mult, dot=black, anchor=leftleg, span=1.5] (m') at (m.top) {};
          \node[comult, anchor=bottom] (c') at (m'.top) {};
          \idstack[anchor=c'.leftleg, direction=up]{2}
          \draw (c.leftleg) -- (c.leftleg |- id2.top);
          \node[twfrobcup, anchor=leftleg] (cup) at (m.rightleg) {};
          \idstack[anchor=cup.rightleg, direction=up]{1}
          \node[twfrobcap, anchor=leftleg] (cap) at (c'.rightleg) {};
          \node[draw=red, dashed, fit={(m'.north) (m'.rightleg) (m.leftleg) (cup.bottom)}]{};
          \idstack[anchor=cap.rightleg, direction=down]{5}
        \end{tikzpicture}
        \xRightarrow{\textcolor{red}{\RTr}}
        \begin{tikzpicture}[baseline=(current bounding box.center)]
          \node[comult] (c) {};
          \node[comult, anchor=bottom] (c') at (c.rightleg) {};
          \node[twfrobcap, anchor=leftleg] (cap) at (c'.rightleg) {};
          \idstack[anchor=c.leftleg, direction=up]{3}
          \idstack[anchor=c'.leftleg, direction=up]{2}
          \idstack[anchor=cap.rightleg, direction=down]{2}
          \node[draw=red, dashed, fit={(c.leftleg) (c.bottom) (c'.rightleg)}]{};
        \end{tikzpicture}
        \overset{\textcolor{red}{\alpha}}{\cong}
        \begin{tikzpicture}[baseline=(current bounding box.center)]
          \node[comult] (c) {};
          \idstack[anchor=c.bottom, direction=down]{1}
          \node[comult, anchor=leftleg] (c') at (id1.bottom) {};
          \node[twfrobcap, anchor=leftleg] (cap) at (c'.rightleg) {};
          \idstack[anchor=cap.rightleg, direction=down]{1}
          \node[draw=red, dashed, fit={(c'.bottom) (id1.bottom)}]{};
        \end{tikzpicture}
        \xRightarrow{\textcolor{red}{\eta_\otimes}}
        \begin{tikzpicture}[baseline=(current bounding box.center)]
          \node[comult] (c) {};
          \idstack[anchor=c.bottom, direction=down]{1}
          \node[comult, anchor=leftleg] (c') at (id1.bottom) {};
          \node[twfrobcap, anchor=leftleg] (cap) at (c'.rightleg) {};
          \idstack[anchor=cap.rightleg, direction=down]{1}
          \node[comult, dot=black, anchor=leftleg, span=1.5] (bc) at (c'.bottom) {};
          \node[mult, anchor=top] (bm) at (bc.bottom) {};
          \node[draw=red, dashed, fit={(c'.leftleg) (cap.top) (cup.rightleg) (bc)}]{};
        \end{tikzpicture}
        \xRightarrow{\textcolor{red}{\RTrPar}}
        \begin{tikzpicture}[baseline=(current bounding box.center)]
          \node[mult] (m) {};
          \node[comult, anchor=bottom] (c) at (m.top) {};
        \end{tikzpicture}
      $}
  \end{equation}
\end{definition}

This is a first step towards a non-symmetric version of the result of
\textcite{hajgatoTracedAutonomousCategories2013}, that every traced symmetric
$*$-autonomous category is autonomous.

  \begin{xlrbox}{epsilon-sup-1}
    \begin{tikzpicture}[baseline=(current bounding box.center)]
      \node[comult] (c) {};
      \node[mult, dot=black, anchor=leftleg] (m) at (c.leftleg) {};
      \node[identity, right=\cobwidth+\cobgap of c] (i) {};
      \node[identity, anchor=bottom] at (i.top) {};
      \node[draw=red, dashed, fit={(c.leftleg) (m.rightleg) (m.top) (c.bottom)}]{};
      \node[draw=blue, dashed, fit={(c.rightleg) (i.top)}]{};
    \end{tikzpicture}
  \end{xlrbox}

  \begin{xlrbox}{epsilon-sup-2}
    \begin{tikzpicture}[baseline=(current bounding box.center)]
      \node[comult] (c) {};
      \node[mult, anchor=leftleg] (m') at (c.rightleg) {};
      \node[comult, dot=black, anchor=bottom] (c') at (m'.top) {};
      \node[mult, dot=black, anchor=rightleg] (m) at (c'.leftleg) {};
      \idstack[anchor=m'.rightleg, direction=down]{1}
      \idstack[anchor=c'.rightleg, direction=up]{1}
      \node[draw=red, dashed, fit={(c'.south) (c'.rightleg) (m.leftleg) (m.top)}]{};
      \node[draw=blue, dashed, fit={(c.bottom) (c.leftleg) (m'.rightleg) (m'.north)}]{};
      \idstack[anchor=m.leftleg, direction=down]{2}
    \end{tikzpicture}
  \end{xlrbox}

  \begin{xlrbox}{epsilon-sup-3}
    \begin{tikzpicture}[baseline=(current bounding box.center)]
      \node[comult] (c) {};
      \node[mult, dot=black, anchor=leftleg] (m) at (c.leftleg) {};
      \node[comult, anchor=bottom] (c') at (m.top) {};
      \node[mult, dot=black, anchor=top] (m') at (c.bottom) {};
      \node[draw=red, dashed, fit={(c.leftleg) (m.rightleg) (m.north) (c.south)}]{};
    \end{tikzpicture}
  \end{xlrbox}

  \begin{xlrbox}{epsilon-prime-sup-1}
    \begin{tikzpicture}[baseline=(current bounding box.center)]
      \node[comult, dot=black] (c) {};
      \node[mult, anchor=leftleg] (m) at (c.leftleg) {};
      \node[identity, left=\cobwidth+\cobgap of c] (i) {};
      \node[identity, anchor=bottom] at (i.top) {};
      \node[draw=red, dashed, fit={(c.leftleg) (m.rightleg) (m.top) (c.bottom)}]{};
      \node[draw=blue, dashed, fit={(c.leftleg) (i.top)}]{};
    \end{tikzpicture}
  \end{xlrbox}

  \begin{xlrbox}{epsilon-prime-sup-2}
    \begin{tikzpicture}[baseline=(current bounding box.center)]
      \node[comult, dot=black] (c) {};
      \node[mult, dot=black, anchor=rightleg] (m') at (c.leftleg) {};
      \node[comult, anchor=bottom] (c') at (m'.top) {};
      \node[mult, anchor=leftleg] (m) at (c'.rightleg) {};
      \idstack[anchor=m'.leftleg, direction=down]{1}
      \idstack[anchor=c'.leftleg, direction=up]{1}
      \node[draw=red, dashed, fit={(c'.south) (c'.leftleg) (m.rightleg) (m.top)}]{};
      \node[draw=blue, dashed, fit={(c.bottom) (c.rightleg) (m'.leftleg) (m'.north)}]{};
      \idstack[anchor=m.rightleg, direction=down]{2}
    \end{tikzpicture}
  \end{xlrbox}

  \begin{xlrbox}{epsilon-prime-sup-3}
    \begin{tikzpicture}[baseline=(current bounding box.center)]
      \node[comult, dot=black] (c) {};
      \node[mult, anchor=leftleg] (m) at (c.leftleg) {};
      \node[comult, dot=black, anchor=bottom] (c') at (m.top) {};
      \node[mult, anchor=top] (m') at (c.bottom) {};
      \node[draw=red, dashed, fit={(c.leftleg) (m.rightleg) (m.north) (c.south)}]{};
    \end{tikzpicture}
  \end{xlrbox}

\begin{proposition}\label[proposition]{delta-inverts}
  Any Cauchy complete left and right $\otimes$-traced $*$-autonomous category
  for which the following equations, along with their symmetric analogues, hold
  is autonomous:

  \begin{minipage}{0.5\textwidth}
    \begin{equation}
        \begin{tikzcd}[math mode=false, sep=tiny, row sep=-25pt, ampersand replacement=\&]
          \xusebox{epsilon-prime-sup-1}
          \arrow[rr, Rightarrow, "\textcolor{red}{$\varepsilon_\otimes$}"]
          \arrow[rd, Rightarrow, "\textcolor{blue}{$\eta_\invamp$}"']
          \&\&
          \xusebox{double-tallid}
          \arrow[rr, Rightarrow, "$\eta_\otimes$"]
          \&\&
          \xusebox{eta-prime-target}
          \\
          \&
          \xusebox{epsilon-prime-sup-2}
          \arrow[rr, Rightarrow, "\textcolor{red}{\labelcref{eq:white-frobenius}}", "\textcolor{blue}{\labelcref{eq:frobenius}}"']
          \&\&
          \xusebox{epsilon-prime-sup-3}
          \arrow[ru, Rightarrow, "\textcolor{red}{$\varepsilon_\otimes$}"', end anchor={south west}]
          \&
        \end{tikzcd}
        \label{eq:epsilon-prime-superposing}\tag{$\varepsilon_\otimes$-sup}
    \end{equation}
    \end{minipage}
    \begin{minipage}{0.49\textwidth}
    \begin{equation}
        \begin{tikzcd}[math mode=false, sep=tiny, row sep=-25pt, ampersand replacement=\&]
          \xusebox{epsilon-sup-1}
          \arrow[rr, Rightarrow, "\textcolor{red}{$\varepsilon_\invamp$}"]
          \arrow[rd, Rightarrow, "\textcolor{blue}{$\eta_\otimes$}"']
          \&\&
          \xusebox{double-tallid}
          \arrow[rr, Rightarrow, "$\eta_\invamp$"]
          \&\&
          \xusebox{eta-target}
          \\
          \&
          \xusebox{epsilon-sup-2}
          \arrow[rr, Rightarrow, "\textcolor{red}{\labelcref{eq:frobenius}}", "\textcolor{blue}{\labelcref{eq:white-frobenius}}"']
          \&\&
          \xusebox{epsilon-sup-3}
          \arrow[ru, Rightarrow, "\textcolor{red}{$\varepsilon_\invamp$}"', end anchor={south west}]
          \&
        \end{tikzcd}
        \label{eq:epsilon-superposing}\tag{$\varepsilon_\invamp$-sup}
    \end{equation}
    \end{minipage}
\end{proposition}

\clearpage
\printbibliography
\label{sec:bibliography}

\clearpage

\end{document}